\documentclass[12pt]{article}
\usepackage{amsmath,amssymb,amsfonts,amsthm}
\usepackage{graphicx}
\usepackage{tikz}
\usepackage{tabularx}
\usepackage[normalem]{ulem}
\usepackage{microtype}

\allowdisplaybreaks

\usepackage{hyperref}

\newcommand*\circled[1]{\tikz[baseline=(char.base)]{\node[shape=circle,draw,inner sep=.8pt] (char) {#1};}}  

\newtheorem{theorem}{Theorem}
\numberwithin{theorem}{section}

\newtheorem{lemma}[theorem]{Lemma}
\newtheorem{proposition}[theorem]{Proposition}

\newtheorem{definition}[theorem]{Definition}
\theoremstyle{remark}\newtheorem{remark}[theorem]{Remark}

\newcommand{\R}{\mathbb{R}}

\newcommand{\Z}{\mathbb{Z}}

\def\SLEkk#1/{$\mathrm{SLE}_{#1}$}
\def\SLEk/{\SLEkk{\kappa}/}
\def\SLEtwo/{\SLEkk2/}
\def\SLE/{$\mathrm{SLE}$}
\def\CLEkk#1/{$\mathrm{CLE}_{#1}$}
\def\CLEk/{\CLEkk{\kappa}/}
\def\CLEtwo/{\CLEkk2/}
\def\CLE/{$\mathrm{CLE}$}
\def\GLEkk#1/{$\mathrm{GLE}_{#1}$}
\def\GLEk/{\GLEkk{\kappa}/}
\def\GLEtwo/{\GLEkk2/}
\def\GLE/{$\mathrm{GLE}$}

\def\Ito/{It\^o}

\def \E {{\bf E}}

\def\Var{\mathrm{Var}}

\def\H{\hspace{.007in}\circled{$\mathrm{H}$}\hspace{.007in}}
\def\C{\hspace{.007in}\circled{$\mathrm{C}$}\hspace{.007in}}
\def\h{\hspace{.007in}\framebox[12pt]{$\mathrm{H}$}\hspace{.007in}}
\def\ch{\hspace{.007in}\framebox[12pt]{$\mathrm{C}$}\hspace{.007in}}
\def\x{\hspace{.007in}\framebox[12pt]{$\mathrm{F}$}\hspace{.007in}}

\def\cou{{\mathcal C}}
\def\dis{{\mathcal D}}

\title{
Quantum gravity and inventory accumulation}
\author{{\sc Scott Sheffield\thanks{Massachusetts Institute of Technology.
Partially supported by NSF grants DMS 064558 and
OISE 0730136.}}
\\
{\it 77 Massachusetts Ave., Room 2-180}\\
{\it Cambridge, MA 02139}\\
{\it USA} }

\date{}

\begin{document}
\maketitle
\begin{abstract}
We begin by studying inventory accumulation at a LIFO (last-in-first-out) retailer with two products.  In the simplest version, the following occur with equal probability at each time step: first product ordered, first product produced, second product ordered, second product produced.  The inventory thus evolves as a simple random walk on $\mathbb Z^2$.  In more interesting versions, a $p$ fraction of customers orders the ``freshest available'' product regardless of type.  We show that the corresponding random walks scale to Brownian motions with diffusion matrices depending on $p$.

We then turn our attention to the {\em critical Fortuin-Kastelyn random planar map model}, which gives, for each $q>0$, a probability measure on random (discretized) two-dimensional surfaces decorated by loops, related to the $q$-state Potts model.  A longstanding open problem is to show that as the discretization gets finer, the surfaces converge in law to a limiting (loop-decorated) random surface.  The limit is expected to be a {\em Liouville quantum gravity} surface decorated by a {\em conformal loop ensemble}, with parameters depending on $q$.  
Thanks to a bijection between decorated planar maps and inventory trajectories (closely related to bijections of Bernardi and Mullin), our results about the latter imply convergence of the former in a particular topology.  A phase transition occurs at $p = 1/2$, $q=4$.
\end{abstract}

\newpage

\section{Introduction}
A planar map is a connected planar graph together with an embedding into the sphere (which we identify with the complex plane $\mathbb C \cup \{ \infty \}$), defined up to topological deformation.  Self loops and multi-edges are allowed.

\begin {figure}[!ht]
\begin {center}
\includegraphics [width=1.5in]{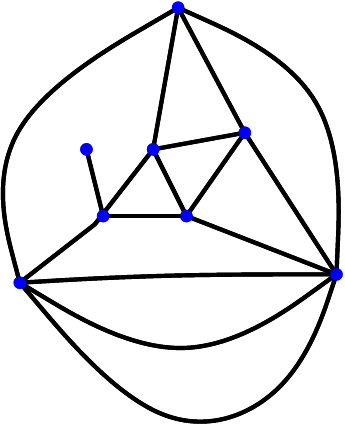}
\end {center}
\end {figure}

There is a significant literature on {\em random} planar maps of various types.  As we illustrate later, planar maps are in one-to-one correspondence with planar quadrangulations, which in turn can be interpreted as Riemannian surfaces, obtained by gluing together unit squares along their boundaries (see Figure \ref{planarmapfigs}).  If we also declare a random subset of the edges of a planar map to be ``open'', we can generate a set of loops on the Riemannian surface that separate open clusters from dual clusters (see Figures \ref{planarmapfigstree} and \ref{planarmapfigsloops}).  We are interested in the ``scaling limits'' of these random loop-decorated surfaces when the number of unit squares tends to infinity.



However, we approach the problem from an unconventional direction.  We begin in Sections \ref{s.inventory} and \ref{s.inventoryproofs} by stating and proving a theorem about inventory accumulation at a LIFO (last-in-first-out) retailer with two products.  Given a certain model for production and sales, in which a $p$ fraction of customers always ``orders fresh'' (i.e., takes the most recently added product regardless of type), we find that the time evolution of the two product inventories scales to a two-dimensional Brownian motion with a diffusion matrix depending on $p$.  (See Theorem \ref{t.brownianlimit} and Figure \ref{burgerplots}.)  We view this result as interesting in its own right.  We also find a surprising phase transition: when $p \geq 1/2$, there are no macroscopic inventory {\em imbalances}.  That is, while the total inventory fluctuates by about $\sqrt n$ after $n$ steps, the {\em difference} between the two product inventories fluctuates by $o(\sqrt n)$.

Section \ref{s.bijection} then presents a bijection between inventory accumulation trajectories and possible instances of so-called critical Fortuin-Kasteleyn (FK) random planar maps.  These are random planar maps --- together with random distinguished edge subsets --- whose laws depend on a parameter $q \in (0,\infty)$ (closely related to the $q$-state Potts model), which turns out to be related to $p$ by $p = \frac{\sqrt q}{2 + \sqrt q}$.  As mentioned above, we may interpret these maps as loop-decorated surfaces.  From this point of view, the inventory-trajectory central limit theorem mentioned above will become a scaling limit theorem about the corresponding loop-decorated surfaces.  Specifically, it will describe (an infinite volume version of) the limiting law of a certain tree and dual tree that are canonically associated to a discretized loop-decorated surface and that encode the structure of the surface.  Our bijection specializes to a classical bijection of Mullin when the distinguished edge subset is required to form a spanning tree of the planar map \cite{MR0205882,MR2285813}.  It is equivalent to a construction by Bernardi when the structure of the planar map is fixed (see Section 4 of \cite{MR2438581}).\footnote{Precisely, Bernardi constructs a many-to-one map from the set of planar maps with distinguished edge subsets to the set of planar maps with distinguished spanning trees and provides a way to understand the pre-image of a {\em single} planar map with a distinguished spanning tree.  Mullin's bijection gives a way of enumerating the set of {\em all} planar maps with distinguished spanning trees.  At the end of \cite{MR2428901}, Bernardi suggests (as a future research project) trying to combine his construction with one of the available enumerations of spanning-tree-decorated planar maps, such as Mullin's bijection.  This is essentially what we do: our inventory trajectory encoding can be understood as a particularly simple way to unify the Mullin and Bernardi constructions.}

The proofs in this paper are discrete and elementary, relying only on standard results in probability (such as the optional stopping theorem).
We remark, however, that we view this paper as part of a larger program to relate random planar maps to continuum objects such as Liouville quantum gravity and the Schramm-Loewner evolution.
We will not discuss these ideas in the body of the paper, but we provide a brief explanation of this point of view in an appendix.  Let us note that since this paper was first posted to the arXiv, there have been many additional papers that have used the theory developed here to prove the convergence the random planar map models to continuum models (involving conformal loop ensembles and Liouville quantum gravity) in certain topologies \cite{dms2014mating, quantum_spheres, gwynnemaosun, finitevolumeestimates, strongertopology, finitevolumelimit}. Other works have described tail exponents of the planar map models \cite{berestyckilaslierray}, a generalization to more than two burger types \cite{multiburger}, a further study of the infinite critical-FK random map \cite{2015arXiv150201013C}, and a related sandpile result \cite{sunwilsonsandpile}.

We also remark that there is a vast operations research literature on inventory management protocols, including various schemes involving two products, but to our knowledge this is the first paper to address this particular model, and also the first to make the connection between inventory trajectories and planar maps.

\medbreak {\noindent\bf Acknowledgments.}  We thank Olivier Bernardi, Bertrand Duplantier, Ewain Gwynne, Nike Sun, Xin Sun, and Sam Watson for very helpful comments on drafts of this paper. We thank an anonymous referee for proposing many improvements that were incorporated into the text.

\section{Inventory trajectories: setup and theorem statement} \label{s.inventory}
In this section we describe a random walk on a particular semi-group (related, as we will see in Section \ref{s.bijection}, to random planar maps) and study its scaling limit.  We interpret the walk as a simple model for inventory accumulation at a LIFO (last-in-first-out) retailer with two product types and three order types (first product, second product, and ``flexible''/``freshest available'') arriving at random times.  As a convenient mnemonic, we refer to the two products as {\em hamburgers} and {\em cheeseburgers} (or {\em burgers} collectively).\footnote{We use this particular metaphor in part for linguistic simplicity (one noun, two prefixes) and in part because the chute stacked with hamburgers and cheeseburgers (produced in a back room, ordered in a front room) is something of an inventory management icon.
}

\subsection{Defining the semigroup} \label{ss.semigroup}

Write $\Theta = \{ \C, \H, \ch, \h, \x\}$.  We view $\Theta$ as an alphabet of symbols that represent, respectively, a cheeseburger, a hamburger, a cheeseburger order, a hamburger order, and a ``flexible'' (either hamburger or cheeseburger) order.  Informally, a word $W$ in $\Theta$ describes a day in the life of a restaurant: for example, the word $\C\H\ch\ch\x\ch\C\h\H$ above describes a day in which first someone produced a cheeseburger (to put on the top of a ``stack'' of burgers), then someone produced a hamburger, then two people ordered cheeseburgers, one ordered ``freshest available,'' another ordered a cheeseburger, someone produced a cheeseburger, someone ordered a hamburger, and someone produced a hamburger.  Informally, we say that whenever a burger is produced, it is added to the top of the stack, and whenever an order is placed it is fulfilled (if possible) by removing the highest matching burger from the stack (but an order that cannot be filled immediately remains unfilled---because the impatient customer leaves the restaurant, say).

To describe this formally, we view the elements of $\Theta$ as generators for a certain (associative) semigroup $\mathcal G$, each element of which can be represented by a word $W$ in the alphabet $\Theta$ (with the empty string $\emptyset$ as a left and right identity).  This is the semigroup of words modulo four ``order fulfillment'' relations
$$\C\ch = \H\h = \C\x = \H\x = \emptyset,$$
and two ``commutativity'' relations
$$\C\h = \h\C,$$ \hspace{.2in} $$\H\ch = \ch\H.$$  A word $W$ in $\Theta$ is called {\em reduced} if no element of $\{\C, \H \}$ appears to the left of any element of $\{ \ch,\h,\x \}$.  In other words, it contains a list of orders followed by a list of burgers, e.g.,
$$W = \ch\h\h\x\h\x\H\C\H\H.$$

\begin{proposition} \label{p.reducedword}
For each finite-length word $W$ in $\Theta$ there is a unique reduced word $\overline W$ that is equivalent to $W$ modulo the above relations.
Thus, $\mathcal G$ is in one-to-one correspondence with the set of reduced words.
\end{proposition}

\begin{proof} Before establishing the uniqueness part of Proposition \ref{p.reducedword}, we establish existence by noting that for each $W \in \Theta$ we can explicitly construct a specific reduced word $\overline W$ inductively as follows.  If $W$ has length zero or one, then we take $\overline W = W$.  Now suppose $\overline W$ has been defined for all $W$ of length $k$ or less.  Then we extend the definition to words of length $k+1$ as follows.  Write $\overline{W \C} = \overline W \C$ and $\overline{W \H} = \overline W \H$, and let $\overline{W \ch}$, $\overline{W \h}$ or $\overline{W \x}$ be obtained by excising (respectively) the rightmost $\C$, $\H$, or ``either $\C$ or $\H$'' symbol from $\overline W$ --- or, if no such symbols exists, appending (respectively) a $\ch$, $\h$ or $\x$ to the right of the list of orders in $\overline W$.

To give a concrete example, consider a map $X: \mathbb Z \to \Theta$.  For $a \leq b$ we will write
 $$X(a,b) := \overline{X(a)X(a+1)X(a+2) \ldots X(b)}.$$
The following illustrates the sequences $X(1,k)$ corresponding to a particular $X(k)$ sequence:

\begin{center}
\begin{tabular}{llr|l}
$X(1) = \,$\C & $ X(1,1) = $ & &\C \\
$X(2) = \,$\H & $ X(1,2) = $ & & \C\H \\
$X(3) = \,$\ch & $ X(1,3) = $ & & \H \\
$X(4) = \,$\ch & $ X(1,4) = $ & \ch & \H \\
$X(5) = \,$\x & $ X(1,5) = $ & \ch &  \\
$X(6) = \,$\ch & $ X(1,6) = $ & \ch \ch & \\
$X(7) = \,$\C & $ X(1,7) = $ & \ch \ch & \C \\
$X(8) = \,$\h & $ X(1,8) = $ & \ch \ch \h & \C \\
$X(9) = \,$\H & $ X(1,9) = $ & \ch \ch \h & \C \H \\
\end{tabular}
\end{center}

If we interpret $W=X(1)X(2) \ldots X(n)$ as a ``day in the life of a restaurant'' then the corresponding $\overline W = X(1,n)$ (in this case $\ch\ch\h\C\H$) contains the unfulfilled orders (in the order they were added --- again, we assume that customers leave the restaurant without a burger if their orders cannot be fulfilled immediately) followed by the unconsumed burgers (in the order they were added).  One may imagine the unconsumed burgers to be arranged in a ``stack'' or a ``chute''.\footnote{
An Internet search for ``burger chute'' turns up many images of multi-lane chutes (which allow burgers to be sorted by type). We stress that our model envisions a {\em single-lane} chute in which each burger is added at the top (so that the order of the remaining burgers corresponds to the order in which they were added) and each order is filled from the top (by taking away the highest suitable burger and letting the remaining burgers slide down).}

Next we claim that if $W_1$ and $W_2$ are equivalent words (modulo the order fulfillment and commutativity relations (i.e., they correspond to the same element of $\mathcal G$) then we will have $\overline{W_1} = \overline{W_2}$.  To see this, observe that inserting a consecutive pair of the form $\C \ch$, $\H \h$, $\C \x$ or $\H \x$ into a word $W$ has no effect on $\overline{W}$, and also that reversing the order of an adjacent pair of the form $\C \h$ or $\H \ch$ in $W$ has no effect on $\overline W$.  This (and the fact that $\overline W = W$ when $W$ is reduced) establishes the uniqueness in Proposition \ref{p.reducedword} and implies that $\mathcal G$ is in one-to-one correspondence with the set of reduced words.
\end{proof}

We also note that the procedure used to construct the reduced form of $X(1,n)$ inductively from $X$ has a ``reverse direction'' analog.  Consider the following example

\begin{center}
\begin{tabular}{llr|l}
$X(0) \,\,\,= \,\,$\ch & $ X(0,0) = $ & \ch & \\
$X(-1) = \,$\H & $ X(-1,0) = $ & \ch & \H \\
$X(-2) = \,$\H & $ X(-2,0) = $ & \ch & \H \H \\
$X(-3) = \,$\C & $ X(-3,0) = $ &    & \H \H \\
$X(-4) = \,$\x & $ X(-4,0) = $ & \x & \H \H  \\
$X(-5) = \,$\ch & $ X(-5,0) = $ & \ch \x & \H \H \\
$X(-6) = \,$\H & $ X(-6,0) = $ & \ch & \H \H \\
$X(-7) = \,$\C & $ X(-7,0) = $ & & \H \H \\
$X(-8) = \,$\C & $ X(-8,0) = $ & & \C \H \H \\
\end{tabular}
\end{center}

From this point of view, when we add a new $X(-k) \in \{ \ch, \h, \x \}$, we are adding this element to the left of the stack of orders, and when we add a new $X(-k) \in \{\C, \H \}$, we are annihilating the leftmost matching order (or adding the burger to the left of the list of unconsumed burgers, if there is no available matching order).  The sequence $X(1,k)$ has a ``burger stack'' whose length can go up and down as $k$ increases but an ``order stack'' whose length can only increase as $k$ increases; similarly, the sequence $X(-k,0)$ has an order stack whose length can go up and down as $k$ increases but a burger stack whose length can only increase as $k$ increases.

\subsection{Defining matches in random infinite words}

Let $X(n)$ be i.i.d.\ random variables, indexed by $n \in \Z$, each of
which takes values in $\{\C,\H,\ch,\h,\x \}$ with respective probabilities $\{1/4, 1/4, (1-p)/4, (1-p)/4, p/2 \}$ for some fixed $p \in[0,1]$ (so that a $p$ fraction of the orders are of type $\x$).  Let $\mu$ denote the corresponding probability measure on
the space $\Omega$ of maps from $\Z$ to $\Theta$.


If a burger that is added at time $m$ is consumed at time $n$, we will say that $m$ and $n$ are a {\em match} and write $m=\phi(n)$ and $n = \phi(m)$.  In other words, if we consider $X(m') \ldots X(n')$ to represent a ``day in the life of the restaurant'' as described above for any $m' \leq m$ and $n' \geq n$, then $m$ and $n$ are a match if and only if the burger added at time $m$ is consumed at time $n$ on that day.  This definition does not depend on the values of $X(k)$ for $k < m$ or $k > n$, since it is equivalent to the statement that $X(m) \in \{\C, \H \}$ and $n$ is the smallest integer greater than $m$ for which $X(m,n)$ does not contain a burger of the same type as $X(m)$.  This in turn holds if and only if one of the following four things happens:

\begin{enumerate}
\item $X(m) = \H$, $X(n) = \h$, and $X(m+1,n-1)$ is a word containing only the letters $\C$ and $\ch$, e.g.,
$$X(m+1, n-1) = \ch\ch\ch\ch\C\C\C\C\C.$$
\item $X(m) = \C$, $X(n) = \ch$, and $X(m+1,n-1)$ is a word containing only the letters $\H$ and $\h$, e.g.,
$$X(m+1, n-1) = \h\h\h\h\H\H\H.$$
\item $X(m) = \H$, $X(n) = \x$ and $X(m+1, n-1)$ is a word containing only the letter $\ch$, e.g.,
$$X(m+1, n-1) = \ch\ch\ch\ch.$$
\item $X(m) = \C$, $X(n) = \x$ and $X(m+1, n-1)$ is a word containing only the letter $\h$, e.g.,
$$X(m+1, n-1) = \h\h\h\h\h\h\h\h.$$
\end{enumerate}

If $X(m) \in \{\C,\H\}$ and $m$ has no match (i.e., the burger added at time $m$ is never consumed --- which would be the case, for example, if we had $X(k) \in \{\C,\H\}$ for all $k > m$) --- we write $\phi(m) = \infty$.  If $X(n) \in \{\ch,\h,\x\}$ and $n$ has no match (i.e., the order at time $n$ is unfulfilled, no matter how far back in time one starts) then we write $\phi(n) = -\infty$.

For example, in the sequence \begin{center}
\begin{tabular}{llr|l}
$X(1) = \,$\C & $ X(1,1) = $ & &\C \\
$X(2) = \,$\H & $ X(1,2) = $ & & \C\H \\
$X(3) = \,$\ch & $ X(1,3) = $ & & \H \\
$X(4) = \,$\ch & $ X(1,4) = $ & \ch & \H \\
$X(5) = \,$\x & $ X(1,5) = $ & \ch &  \\
$X(6) = \,$\ch & $ X(1,6) = $ & \ch \ch & \\
$X(7) = \,$\C & $ X(1,7) = $ & \ch \ch & \C \\
$X(8) = \,$\h & $ X(1,8) = $ & \ch \ch \h & \C \\
$X(9) = \,$\H & $ X(1,9) = $ & \ch \ch \h & \C \H \\
\end{tabular}
\end{center}
described above, we have $\phi(3) = 1$; $\phi(1)=3$ and $\phi(2) = 5$; $\phi(5) = 2$, but the values of $\phi(4), \phi(6), \phi(7), \phi(8), \phi(9)$ necessarily lie outside the interval $\{1,2,\ldots,9\}$ and are not determined by $X(1), X(2), \ldots, X(9)$.

\begin{proposition} \label{p.allmatched} It is $\mu$ almost surely the case that for every $m \in \mathbb Z$, we have $\phi(m) \not\in \{-\infty, \infty\}$.  In other words, every $X(j)$ has a unique match, almost surely, so that $\phi$ is an involution on $\mathbb Z$.
\end{proposition}

\begin{proof} We first claim that this holds whenever $X(m) \in \{\C,\H \}$.  Observe that the net number of burgers (i.e., the number of burger symbols minus the number of order symbols) added between times $m$ and $n$, as a function of $n$, is a simple random walk on $\mathbb Z$.  It follows that there will almost surely exist values of $n>m$ for which this quantity is arbitrarily negative, and hence $X(m+1,n)$
contains an arbitrarily long sequence of orders.  (Recall that the number of orders in $X(m+1,n)$ is non-decreasing as $n$ increases.) If $X(m) = \C$, then the first time that a $\ch$ or $\x$ is added to this list of orders will be a time at which $X(m)$ is consumed; thus, on the event that $m$ has no match it is almost surely the case that the number of $\h$ orders in
 $X(m+1,n)$
 tends to infinity as $n \to \infty$ while the number of $\ch$ orders remains zero.  If this happens for some $m$, then there cannot be an $m'$ for which the same thing happens with the roles of hamburgers and cheeseburgers reversed (since $X(m+1,m')$ is a fixed finite length word --- appending it on the left cannot remove an arbitrarily long sequence of $\ch$ elements from the corresponding reduced word).  Thus, it is almost surely the case that either every cheeseburger added (at any integer time) is ultimately consumed or every hamburger added (at any integer time) is ultimately consumed.  Since each of these two events is translation invariant, the zero-one law for translation invariant events (see any textbook with an introduction to ergodic decompositions, e.g., \cite{georgii2011gibbs}) implies that each has probability zero or probability one.  As observed above, the union of these two events has probability one, so by symmetry each of them separately has $\mu$ probability one.  Thus, $\mu$ a.s.\ {\em every} burger of either type is ultimately consumed, which implies the claim.  A similar argument (in the reverse direction) shows that $\phi(n)$ is a.s.\ finite whenever $X(n) \in \{\ch,\h,\x \}$.
\end{proof}

The following is an immediate consequence of the above construction:

\begin{proposition} \label{p.finiteunmatched} The reduced word $X(1,n)$ contains precisely those $X(k)$ corresponding to the $k \in \{1,2,\ldots,n\}$ for which $\phi(k) \not \in \{1,2,\ldots,n\}$.  In other words, it contains the list of unmatched orders (in order of appearance in the sequence $X(1), X(2), \ldots, X(n)$) followed by the ordered list of unmatched burgers (in order of appearance in $X(1), X(2), \ldots, X(n)$).
\end{proposition}

\subsection{Infinite stacks and random walks: main theorem}

By analogy with Proposition \ref{p.finiteunmatched}, we define $X(-\infty,n)$ to be the ordered sequence of $X(k)$ for which $k \leq n$ but $\phi(k) > n$.  Informally, we can write
$$X(-\infty,n) = \overline{\ldots X(-3) X(-2) X(-1) X(0) X(1) X(2) \ldots X(n)},$$
and interpret $X(-\infty,n)$ as the reduced form word corresponding to the product of all $X(k)$ with $k \leq n$.
We view $X(-\infty, n)$ as a semi-infinite stack of $\C$ and $\H$ symbols, indexed by the negative integers; it has a unique top element but no bottom element.  (The number of elements in $X(-\infty,n)$ is $\mu$ a.s.\ infinite for all $n$ by Proposition \ref{p.allmatched} --- if it were some finite value $k$, then with positive probability there would be a consecutive sequence of $k+1$ orders after time $n$, and at least one order would have no match.)

\begin {figure}[!ht]
\begin {center}
\includegraphics [width=4in]{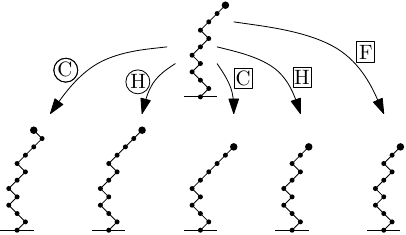}
\caption {\label{stackchange} Above: top few elements on the stack $X(-\infty, n-1)$ represented by an up-left edge for each cheeseburger and an up-right edge for each hamburger.  Below: top few elements of $X(-\infty,n)$, depending on value of $X(n)$.   The uppermost dot above (shown slightly larger) is positioned at a location $A_{n-1} = (\dis_{n-1},\cou_{n-1})$ describing the net change in inventory since time zero. The new uppermost dots below are positioned at locations $A_n = (\dis_n,\cou_n)$.}
\end {center}
\end {figure}

We similarly write
$$X(n,\infty) := \overline{X(n) X(n+1) \ldots},$$
which we interpret to mean the ordered sequence of $X(k)$ for which $k \geq n$ but $\phi(k) < n$.  It is natural to represent $X(n,\infty)$ as a sequence of $\{\ch,\h,\x \}$ values indexed by the positive integers.  While the $X(-\infty,n)$ can be interpreted as a semi-infinite stack of burgers waiting to be consumed (as $n$ increases in time), the stack $X(n,\infty)$ can be interpreted as a semi-infinite queue of customers waiting to be served (as $n$ decreases in time).  A useful equivalent definition is that $X(n,\infty)$ is the limit of the order stacks of $X(n,m)$ (which only increase in length as $m$ increases, with new orders being added on the right), as $m \to \infty$.  Similarly, $X(-\infty,n)$ is the limit of the burger stacks of $X(m,n)$ (which only increase in length as $m$ decreases, with new burgers being added on the left) as $m \to -\infty$.


Next, define $$Y(n) := \begin{cases} X(n) & \hspace{.1in} X(n)  \in \{\C, \H, \ch, \h \} \\ \ch & \hspace{.1in} X(n) = \x, \,\,\, X(\phi(n)) = \C  \\ \h & \hspace{.1in} X(n) = \x, \,\,\, X(\phi(n)) = \H \end{cases}.$$  In other words, $Y(n)$ is obtained from $X(n)$ by replacing each $\x$ with a $\ch$ or $\h$, depending on which burger type was actually consumed by the $\x$ order.  For $a \leq b$ we also write
$$Y(a,b) := \overline{Y(a)Y(a+1)\ldots Y(b)},$$
and observe that this is the same as $X(a,b)$ except that each $\x$ is replaced with the corresponding $\ch$ or $\h$ symbol.

For every word $W$ in the symbols $\{\C,\H,\ch,\h, \x \}$ we write $\cou(W)$ for the net burger {\em count} (i.e., the number of $\{\C,\H\}$ symbols minus the number of $\{\ch,\h,\x \}$ symbols in $W$).  Analogously, if $W$ has no $\x$ symbols, then we define $\dis(W)$ to be the net {\em discrepancy} of hamburgers over cheeseburgers (i.e., the number of $\{\H,\ch\}$ symbols minus the number of $\{\C,\h\}$ symbols).

\begin{definition} \label{andef} Given the infinite $X(n)$ sequence, let $\cou_n$ be the integer valued process defined by $\cou_0 = 0$ and $\cou_n - \cou_{n-1} = \cou(Y(n))$ for all $n$.  Similarly write $\dis_0=0$ and $\dis_n - \dis_{n-1} = \dis(Y(n))$.  Thus $\cou_n$ and $\dis_n$ keep track of the net change in the burger count and the burger discrepancy since time zero.  When $n \geq 0$ we have $\cou_n = \cou(Y(1,n))$ and $\dis_n = \dis(Y(1,n))$.  (For this purpose, we will write $Y(1,0) = \emptyset$ by convention.) When $n < 0$ we have $\cou_n = -\cou(Y(n+1,0))$ and $\dis_n = -\dis(Y(n+1),0)$.  As a shorthand and slight abuse of notation, we will also write, when $a$ and $b$ are integers,
$$\cou(a) = \cou(Y(a)), \,\,\,\, \cou(a,b) = \cou(Y(a,b)), \,\,\,\, \dis(a) = \dis(Y(a)), \,\,\,\, \dis(a,b) = \dis(Y(a,b)).$$
We also write $A_n = (\dis_n, \cou_n)$ for integer $n$.  We extend the definition to the real numbers (by piecewise linear interpolation) so that $t \to A_t$ is an infinite continuous path.
\end{definition}

Figure \ref{stackchange} suggests a natural way to visualize the evolution of $X(-\infty,n)$: here $X(-\infty,n)$ is represented by a stack of up-left edges (one for each $\C$) and up-right edges (one for each $\H$) ending at the location $A_n$.
If $p=1$, then every time $\cou_n$ decreases, it backtracks along $X(-\infty,n)$, and every time $\cou_n$ increases, the corresponding change in $\dis_n$ is decided by a coin toss.  (This is related to a process called the {\em Brownian snake}, see e.g.\ \cite{MR1339740}.)  In this case, for each fixed $n$, the burger stack $X(-\infty,n)$ has the law of a sequence of i.i.d.\ elements of $\{\C, \H \}$.  When $p=0$, the process $A_n$ is just a simple random walk on the set of integer pairs $(x,y)$ for which $x+y$ is even, and the stack $X(-\infty,n)$ encodes information about the past history of $A_n$.  (Informally, the ``head of the snake'' moves around and drags the body along with it in a way that maintains the requirement that the body always consists of a sequence of up-left and up-right edges, as in Figure \ref{stackchange}.)  When $p \in (0,1)$ the transition probabilities are averages of the extreme cases $p=0$ and $p=1$.

Our main scaling limit result concerns the random process $A_t$ of Definition \ref{andef}.

\begin{theorem} \label{t.brownianlimit}
As $\epsilon \to 0,$ the random functions $\epsilon A_{t/\epsilon^2}$ converge in law (with respect to the $L^\infty$ metric on any compact interval) to \begin{equation} \label{e.brownianlimit} (B^1_{\alpha t}, B^2_{t}), \end{equation} where $B^1_t$ and $B^2_t$ are independent standard one-dimensional Brownian motions and $$\alpha = \max\{1-2p,0 \}.
$$
\end{theorem}

\begin {figure}[!ht]
\begin {center}
\includegraphics [height=4in]{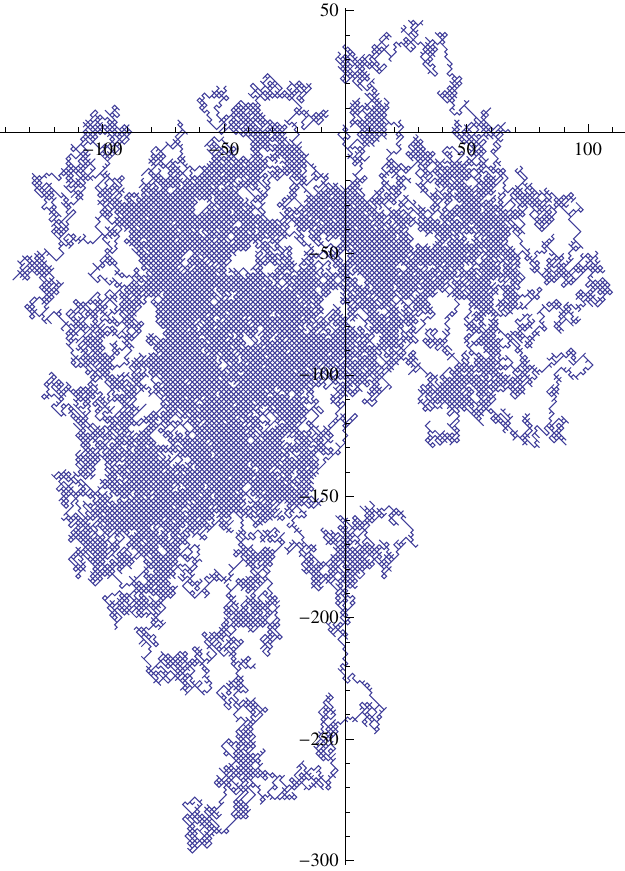}\hspace{.1in}
\includegraphics [height=6in]{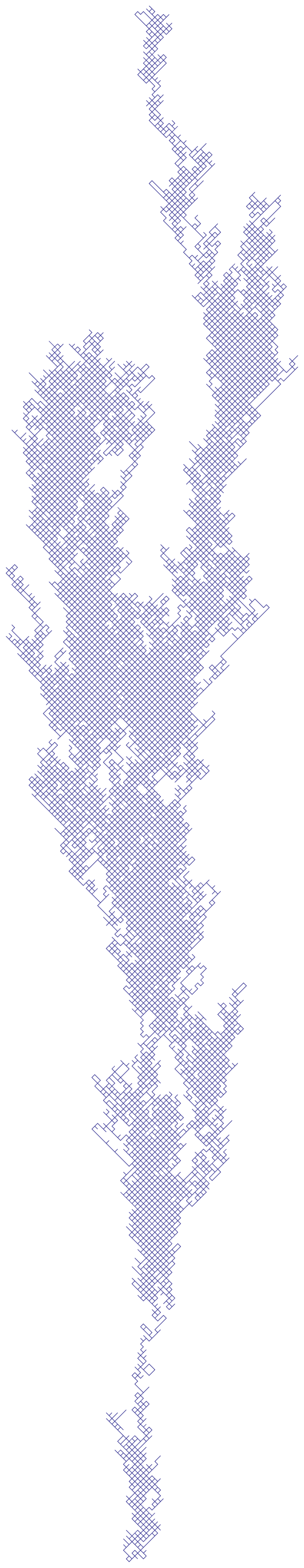}\hspace{.1in}
\includegraphics [height=5in]{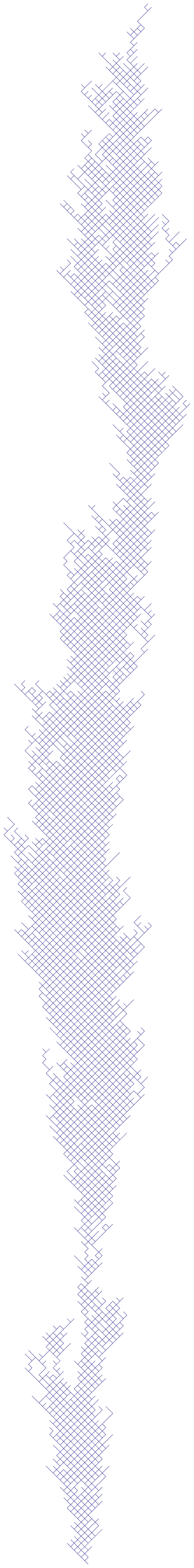}
\caption {\label{burgerplots} Sample trajectory traced by $A_n = (\mathcal D_n, \mathcal C_n)$ as $n$ ranges from $0$ to $100000$ when $p=.25$ (left).  Similar trajectories (axes not shown) when $p=.75$ (middle) and $p=1$ (right).  In each case we use the same initial data: $X(-\infty,0)$ is an infinite stack alternating between burger types.  In each case $X(1), \ldots, X(100000)$ are then chosen from $\mu$ and $(\mathcal D_n, \mathcal C_n)$ updated accordingly for $n \in \{1,2,\ldots, 100000\}$.  As $n$ tends to infinity, and both axes are scaled by $1/\sqrt{n}$, the left trajectory tends to a two dimensional Brownian motion with variance $1-2p = .5$ times as large in the left-right direction as the up-down direction.  The middle and right trajectories tend to one-dimensional Brownian motions, concentrated on a vertical axis.  (See Remark \ref{r.generalization}.)}
\end {center}
\end {figure}

In particular, the above theorem implies that $\cou_n$ scales to ordinary Brownian motion, which is not surprising since it is just a simple random walk on $\mathbb Z$ regardless of $p$.  When $p=0$, the processes $\dis_n$ and $\cou_n$ are independent simple random walks on $\mathbb Z$, and it is also unsurprising that $A_n=(\dis_n, \cou_n)$ scales to an ordinary two-dimensional Brownian motion in this case.  When $p=1$ (and one only removes burgers from the top of the stack, because all orders are flexible), the law of $X(-\infty,n)$ (for any fixed $n$) is that of an i.i.d.\ sequence of $\C$ and $\H$ values.  In this case, the fact that inventory changes tend to be ``well balanced'' between hamburgers and cheeseburgers (so that the first term in \eqref{e.brownianlimit} is identically zero and $\epsilon A_{t/\epsilon^2}$ concentrates on the vertical axis as $\epsilon \to 0$) can be deduced from the law of large numbers.  Indeed, when $p=1$, the magnitude of $\cou_n$ has order $n^{1/2}$ with high probability, while the magnitude of $\dis_n$ has order $n^{1/4} = o(n^{1/2})$.  To see where the $n^{1/4}$ comes from, recall that $\cou_0 = 0$ and condition on $m := \min\{ \cou_j : 0 \leq j \leq n \}$ and on $\cou_n$, noting that $\cou_0 - m$ and $\cou_n - m$ are both of order $\sqrt{n}$.  Note that the top $\cou_0 - m$ burgers in the time zero stack fill the orders remaining in the reduced word $X(1,n)$, and the top $\cou_n - m$ burgers in the time $n$ stack are the types of the burgers in $X(1,n)$.
The types of the top $\cou_0 - m$ burgers in the time zero stack and the top $\cou_n - m$ burgers in the time $n$ stack can then be determined with independent coin tosses, so that $\dis_n$ is of order $\sqrt{\sqrt{n}} = n^{1/4} = o(n^{1/2})$.

The theorem states that as long as $p \geq 1/2$ this balance continues to hold: when $n$ is large, the net inventory accumulation between time zero and $n$ is close to evenly divided between hamburgers and cheeseburgers, with high probability.  When $p < 1/2$ the fluctuations of $\dis_n$ are on the same order as those of $\cou_n$.  Put differently, a LIFO retailer accumulates major inventory discrepancies (on the same order as the total inventory fluctuation) if and only if more than half of its customers have a product preference.  Figure \ref{burgerplots} illustrates sample trajectories of $A_n$ for both $p<1/2$ and $p>1/2$.

\section{Inventory trajectories: constructions and proofs}\label{s.inventoryproofs}

The primary goal of this section is to prove Theorem \ref{t.brownianlimit}.  On the way, we will establish some independently interesting lemmas involving properties of ``excursion words'' of various types, equivalent formulations of the $p=1/2$ phase transition, typical lengths of random reduced words, and monotonicity properties for the time evolution of an inventory stack.

We recall that the {\bf optional stopping theorem} states that if $X_0, X_1, X_2, \ldots$ is a martingale (resp., supermartingale) and $K$ is a bounded stopping time then $\mathbb E[X_K] = \mathbb E[X_0]$ (resp., $\mathbb E[X_K] \leq \mathbb E[X_0]$).  Furthermore, if $K$ is an arbitrary (possibly unbounded) stopping time and $X_1, X_2, \ldots$ is a martingale (or supermartingale) {\em whose values are bounded below} then $\mathbb E[X_K] \leq \mathbb E[X_0]$. By symmetry, this also implies that if $K$ is an arbitrary stopping time and $X_1, X_2, \ldots$ is a martingale (or submartingale) whose values are bounded above then $\mathbb E[X_K] \geq \mathbb E[X_0]$. We will use these facts frequently.


\subsection{A variance calculation} \label{ss.variance}

Let $J$ be the smallest positive integer for which the (reduced) word $X(-J,-1)$ has a non-empty burger stack (i.e., at least one $\C$ or $\H$ symbol) and let $\chi = \chi(p)$ be the expected length of $X(-J,-1)$.  As preparation for proving Theorem \ref{t.brownianlimit}, we will prove the following:
\begin{lemma} \label{l.pchivariance}
If $p \leq 1/2$ then $\chi=2$ and $\Var[\dis_n] = (1-2p) n + o(n)$, while if $p \geq 1/2$ then $\chi=1/p$ and $\Var[\dis_n] = o(n)$, as summarized in the table below:
\begin{center}
\begin{tabular} {|l|l|l|}
\hline
$p \leq 1/2$ & $\chi = 2$ & $\lim_{n \to \infty} n^{-1} \Var[\dis_n] = 1-2p$ \\
\hline
$p \geq 1/2$ & $\chi = 1/p$ & $\lim_{n \to \infty} n^{-1} \Var[\dis_n] = 0$ \\
\hline
\end{tabular}
\end{center}
\end{lemma}
\noindent{This section will prove Lemma \ref{l.pchivariance} modulo one other lemma (Lemma \ref{l.chigives}) to be established later.  The rough idea behind the argument is pretty simple. On the event that $X(0) = \x$, it is not hard to see that $\phi(0) = -J$, as defined above. Our intuition suggests that the addition of extra $\x$ symbols to the model should decrease the fluctuation of inventory imbalances.  Basically, this is because on the event that a fresh order at time $0$ consumes a cheeseburger (i.e., $X(0) = 0$ and $X(-J) = \C$), one might guess that, on average, there had been some net surplus of $\C$ during the period immediately before time $0$. In particular, one would expect the word $X(-J,-1)$ to contain a net surplus on $\C$ symbols. Quantifying this effect precisely will amount to computing the expectation of $\dis(-J,-1)$ on the event $\phi(-J) = \C$, which will turn out (see \ref{i::orderlength} in the list below) to be equivalent to computing $\chi = \mathbb E[|X(-J,-1)|]$. The arguments below will follow this idea and will give the reader a quick idea of
where the $\alpha = \max \{1-2p,0 \}$ of Theorem \ref{t.brownianlimit} comes from and why a phase transition occurs at $p=1/2$.  We present them as a series of observations:}
\begin{enumerate}
 \item \label{i::orderlength} The order stack of $X(-J,-1)$ must consist of zero or more orders of type opposite to the one burger in $X(-J,-1)$, e.g., $X(-J,-1) = \ch\ch\ch \H$ or $\h\h\C$ or $\C$.  In particular, the length of the word $X(-J,-1)$ is given by \begin{equation}\label{e.diswordlength} |X(-J,-1)| = |\dis(-J,-1)| = -\cou(-J,-1)+2 . \end{equation}
 \item $\cou(-j,-1)$ is a martingale in $j$ satisfying $\cou(-j,-1) \leq 1$ for all $j \in \{1,2,\ldots,J \}$ and $\mathbb E[\cou(-1,-1)] = 0$. By the optional stopping theorem for martingales bounded above, the expectation of a martingale at a (possibly unbounded) stopping time is at least the expectation at time zero (and at most the upper bound). This gives $\mathbb E[\cou(-J,-1)] \in [0,1]$, and hence by \eqref{e.diswordlength}, \begin{equation} \label{e.chidef} \chi := \mathbb E \bigl[|(X(-J,-1)|\bigr] = \mathbb E\bigl[|\dis(-J,-1)|\bigr] \in [1,2]. \end{equation}
\item By \eqref{e.diswordlength} and \eqref{e.chidef}, we have $\mathbb E[\cou(-J,-1)] = \mathbb E[\cou(-1,-1)]=0$ if and only if $\chi=2$.  Since the optional stopping theorem implies
    $$\mathbb E[\cou(-1,-1)] = \mathbb E[\cou(-J,-1) 1_{J \leq n}] + \mathbb E[\cou(-n,-1) 1_{J>n}],$$
    and the former term on the RHS tends to $\mathbb E[\cou(-J,-1)]$ we have
    \begin{equation} \label{e.chi2} \chi = 2 \,\,\,\,\text{if and only if} \,\,\,\,  \lim_{n \to \infty} \mathbb E[\cou(-n,-1) 1_{J > n}] = 0. \end{equation}
 \item By \eqref{e.chidef}, the expectation $\mathbb E[\dis(-J,-1)]$ exists and, by symmetry,
     \begin{equation} \label{e.zeroexpdis} \mathbb E[\dis(-J,-1)] = 0. \end{equation}
 By \eqref{e.chi2} and the fact that $|\dis(-n,-1)| \leq |\cou(-n,-1)| = -\cou(-n,-1)$ if $n<J$, \begin{equation} \label{e.disJ} \chi = 2 \,\,\,\, \text{implies} \,\,\,\,
 \lim_{n \to \infty} \mathbb E\bigl[|\dis(-n,-1)| 1_{J > n}\bigr] = 0. \end{equation}
 \item On the event $X(0) \not = \x$ the value $\dis(0)$ is determined by $X(0)$ independently of $\dis(-J,-1)$.  Hence, by \eqref{e.zeroexpdis}, the expectation of $\dis(0)\dis(-J,-1)$ restricted to this event is zero.  On the event $X(0) = \x$ we have $J = \phi(0)$ and hence $\dis(0)$ has sign opposite that of $\dis(-J,-1)$.  Since $X(0) = \x$ with probability $p/2$ and (even after conditioning on $X(0) = \x$) we have $\mathbb E\bigl[|\dis(-J,-1)|\bigr] = \chi$, this implies \begin{equation} \label{e.disexp} \mathbb E[\dis(0) \dis(-J,-1)] = -\frac{\chi p}{2}.\end{equation}
\item On the event that $J < n$ the expectation of $\dis(0) \dis(-n,-J-1)$ is zero, since $\dis(-n,-J-1)$ is independent of $\dis(0)$ on this event (because $X(-J,0)$ contains no $\x$ symbols, and one can swap the roles of hamburgers and cheeseburgers in the word $X(-J)X(-J+1)\ldots X(0)$ independently of the values of $X_j$ for $j < -J$).  Thus the expectations of $\dis(0) \dis(-n,-1)$ and
    $\dis(0) \dis(-J,-1)$ are the same on the event $J<n$ (and trivially also the same on the event $J=n$).  Thus
\begin{equation} \label{e.disbound} \mathbb E[\dis(0) \dis(-n,-1)] = \mathbb E[\dis(0) \dis(-J,-1) 1_{J  \leq n }] + \mathbb E[\dis(0) \dis(-n,-1) 1_{J > n }].
\end{equation} By \eqref{e.disJ}, the latter term in \eqref{e.disbound} tends to zero as $n \to \infty$ when $\chi = 2$.  Thus,
\begin{equation} \label{e.varwithpast} \chi = 2 \,\,\,\, \text{implies} \,\,\,\, \lim_{n \to \infty} \mathbb E[ \dis(0) \dis(-n,-1)] = -\frac{\chi p}{2}  = -p.\end{equation}
Using translation invariance of the law of $Y_j$ and \eqref{e.varwithpast}, we obtain that
\begin{equation} \label{e.varlimit}\chi = 2 \,\, \text{implies} \,\,\Var[\dis_n] = \sum_{i=1}^n
 \mathbb E [\dis(Y_i)^2] + 2 \sum_{i=2}^n \mathbb E[\dis(i) \dis(1,i-1) ] =
 n  -2 \frac{\chi p}{2} n + o(n),\end{equation} which in particular implies that $1-\chi p \geq 0$ so that \begin{equation} \label{e.p12} \chi = 2 \,\,\,\, \text{implies both} \,\,\,\, p \leq 1/2 \,\,\,\, \text{and} \,\,\,\, \Var[\dis_n] = (1-2p)n + o(n). \end{equation}
\item  Lemma \ref{l.chigives} below states that the conclusion of \eqref{e.disJ} holds even if $\chi \not = 2$.  For now, we note that that the above calculations show that the conclusions of \eqref{e.varwithpast} and hence \eqref{e.varlimit} remain true when $\chi \not = 2$ {\em contingent} on this claim:
\begin{equation} \label{e.varwithpast2}
\lim_{n \to \infty} \mathbb E\bigl[ |\dis(-n,-1)| 1_{J > n }\bigr]=0  \,\,\,\, \text{implies} \,\,\,\,\Var[\dis_n] = n  - \chi p n + o(n).\end{equation}
From this and \eqref{e.chidef} we see the following: \begin{equation*} \Var[\dis_n] = o(n) \,\,\, \text{and} \,\,\,\lim_{n \to \infty} \mathbb E[| \dis(-n,-1)| 1_{J > n }]=0 \,\,\,\, \text{together imply} \end{equation*} \begin{equation}\label{e.chivalue} 1- \chi p = 0 \,\,\,\, \text{and hence} \,\,\,\, \chi = 1/p \,\,\,\, \text{and} \,\,\,\, p \geq 1/2.\end{equation}
 \end{enumerate}

The following lemma will be proved later:

\begin{lemma} \label{l.chigives} If $\chi \not = 2$ then $\Var[\dis_n] = o(n)$ and $\lim_{n \to \infty} \mathbb E[| \dis(-n,-1)| 1_{J > n }]=0$.
\end{lemma}

Lemma \ref{l.chigives} states, in other words, that for each $p \in [0,1]$, {\em either} the left statement in \eqref{e.p12} (that $\chi = 2$) is true {\em or} the left statements in \eqref{e.chivalue} (that $\Var[\dis_n] = o(n)$ and $\lim_{n \to \infty} \mathbb E[| \dis(-n,-1)| 1_{J > n }]=0$) are true.
We now claim that Lemma \ref{l.pchivariance} is a consequence of \eqref{e.disJ}, \eqref{e.p12}, \eqref{e.chivalue}, and Lemma \ref{l.chigives}.
One one hand Lemma \ref{l.chigives} and \eqref{e.chivalue} imply that if $p<1/2$ then $\chi = 2$ and \eqref{e.p12} gives the top row of  Lemma \ref{l.pchivariance}. On the other hand \eqref{e.p12} implies that if $p > 1/2$ then $\chi \not = 2$, and  Lemma \ref{l.chigives} and \eqref{e.chivalue} give the bottom row. If $p = 1/2$, then the top and bottom row are equivalent, and we must have $\chi = 2$ (giving the desired result by \eqref{e.p12}) since $\chi \not = 2$ would lead to a contradiction using Lemma \ref{l.pchivariance} and \eqref{e.chivalue}.

In preparation for proving Lemma \ref{l.chigives}, we will derive several additional consequences of the assumption that $\chi \not = 2$ in Section \ref{ss.excursionwords}.

\subsection{Two simple finite expectation criteria} \label{ss.twocriteria}
This section makes two simple observations that will be useful in Section \ref{ss.excursionwords}.  The first is a special case of what is sometimes called Wald's identity. (We include a short proof for convenience in case the reader has not seen this before.)

\begin{lemma} \label{l.finitelystoppedexpectedsum}
Let $Z_1, Z_2, \ldots$ be i.i.d.\ random variables on some measure space and $\psi$ a measurable function on that space for which $\mathbb E[\psi(Z_1)]$ is well defined and finite.  Let $T$ be a stopping time of the process $Z_1, Z_2, \ldots$ with the property that $\mathbb E[T]$ is finite.  Then $\mathbb E[\sum_{j=1}^T \psi(Z_j)]$ is also well defined and finite.
\end{lemma}
\begin{proof}
It is enough to consider the case that $\psi$ is non-negative (since we can write a general $\psi$ as a difference of non-negative functions).
Since $T$ is a stopping time, we know that for each fixed $j$, the value of $Z_j$ is independent of the event that $T \geq j$. Thus
\begin{equation} \label{eqn::stoppingtotal} \mathbb E[\sum_{j=1}^T \psi(Z_j)]  = \sum_{j=1}^\infty \mathbb E[\psi(Z_j) 1_{T \geq j}] = \sum_{j=1}^\infty \mathbb E[\psi(Z_j)] \mathbb E[1_{T \geq j}] = \mathbb E[\psi(Z_1)] \mathbb E[T].\end{equation}
\end{proof}

\begin{lemma} \label{l.finiteexpectationtozero}
Let $Z_1, Z_2, \ldots$ be i.i.d.\ random variables on some measure space and let $\mathcal Z_n$ be a non-negative-integer-valued process adapted to the filtration of the $Z_n$ (i.e., each $\mathcal Z_n$ is a function of $Z_1, Z_2, \ldots, Z_n$) that has the following properties:
\begin{enumerate}
\item {\bf Bounded initial expectation:}  $\mathbb E[\mathcal Z_1] < \infty$.
\item {\bf Positive chance to hit zero when close to zero:} For each $k > 0$ there exists a positive $p_k$ such that $\mathbb P[\mathcal Z_{n+1} = 0 | Z_1, \ldots, Z_k] \geq p_k$ on $\{\mathcal Z_n = k \}$.
\item {\bf Uniformly negative drift when far from zero:} There exist positive constants $C$ and $c$ such that $\mathbb E[\mathcal Z_{n+1} - \mathcal Z_n | Z_1, \ldots, Z_k] \leq -c$ on $\{\mathcal Z_n \geq C \}$.
\item {\bf Bounded expectation when near zero:} There further exists a constant $b$ such that $\mathbb E[\mathcal Z_{n+1} | Z_1, \ldots, Z_n] < b$ on $\{\mathcal Z_n < C \}$.
\end{enumerate}
Then $$\mathbb E \bigl[ \min \{n : \mathcal Z_n = 0 \} \bigr] < \infty.$$
\end{lemma}
\begin{proof}
The uniformly negative drift assumption implies that the quantity $Y_n = \mathcal Z_n + cn$ is a supermartingale until time $K = \min \{n\geq 0 : \mathcal Z_n < C \}$. In particular this fact together with the optional stopping theorem implies that as a function of $k$ the quantity $\mathbb E[Y_{ K \wedge k}]$ is bounded above independently of $k$.  On the event $K \geq k$ we have $Y_k \geq c k$.  Since $Y_{ K \wedge k}$ is non-negative, this implies that $ ck \mathbb P[ K \geq k] \leq \mathbb E[Y_{ K \wedge k}]$. In particular, this implies that $\lim_{k \to \infty} \mathbb P[ K \geq k] = 0$ so that $K$ is a.s.\ finite.

Next, by the optional stopping theorem for non-negative supermartingales, $\mathbb E[Y_K] \leq \mathbb E[Y_1]$ which implies $\mathbb E[c K] \leq \mathbb E[Y_1]$ and $\mathbb E[K] \leq \mathbb E[Y_1]/c$.  Now let $T_1, T_2, \ldots$ denote the successive times when $\mathcal Z_n < C$. Then we claim that a similar analysis implies that $\mathbb E[T_{k+1} - T_k | Z_1, Z_2, \ldots, Z_{T_k}] \leq \beta$ for some positive constant $\beta$. To carry out this analysis, we will write $\tilde Y_n = \mathcal Z_{T_k+n} + cn$, and note that the same argument as above (using probability conditioned on $Z_1, Z_2, \ldots, Z_{T_k}$ and noting that $\mathbb E[\tilde Y_1 | Z_1, Z_2, \ldots, Z_{T_k}] \leq b+c$) shows that $$\mathbb E[T_{k+1}-T_k| Z_1, Z_2, \ldots, Z_{T_k}] \leq \mathbb E[\tilde Y_1 | Z_1, Z_2, \ldots, Z_{T_k}]/c \leq (C+c)/c =: \beta.$$  Next, recall that whenever $\mathcal Z_n < C$, there is a probability of at least $\delta = \min \{p_k : k \leq C \} > 0$ that $\mathcal Z_{n+1} = 0$. This implies that if $S := \min \{n : \mathcal Z_n = 0 \}$ then $\mathbb P[S > T_k] \leq \delta^{k-1}$.  Since \begin{eqnarray} \mathbb E[S] & \leq & \mathbb \mathbb E[T_1] + E[\sum_{k=1}^{\infty}  (T_{k+1} - T_k) 1_{S > T_k}] \\
& \leq & \mathbb E[K] + \sum_{k=1}^\infty \mathbb P[T_k > S] \mathbb E[T_{k+1} - T_k | S > T_k] \\
& \leq & \mathbb E[K] + \sum_{k=1}^\infty \delta^{k-1} \beta < \infty.
\end{eqnarray}
\end{proof}

\subsection{Excursion words} \label{ss.excursionwords}

In Section \ref{ss.variance} we considered the random word $X(-J,-1)$, where $J$ was the smallest integer for which this word had at least one $\C$ or $\H$ symbol.  We found that the expected word length $\mathbb E[|X(-J,-1)|]$ was a constant $\chi \in [1,2]$. In this section we will consider different words and show that they all have finite expected length provided that $\chi < 2$.

Suppose that $K$ is the smallest $k \geq 0$ for which $\cou_{k+1} = \cou(1,k+1) < 0$.  If $X(1) \in \{\ch, \h, \x \}$ then $K = 0$.  Otherwise, $K$ is a positive value for which $\cou_K = \cou(1,K) = 0$, so that $X(1,K)$ is ``balanced'' in the sense that it has the same number of burgers as orders.  Call $E = X(1,K)$ the {\bf excursion word} beginning at time zero (writing $E = \emptyset$ if $K = 0$).  We make a few observations about $E$:
\begin{enumerate}
\item $E$ a.s.\ contains no $\x$ symbols.
   (Indeed, one can check inductively that $X(1,k)$ contains no $\x$ symbols for any $1 \leq k \leq K$.)
\item If $p = 1$ then $E = \emptyset$ almost surely.
\item The law of $K$ is independent of $p \in [0,1]$.  Its law is that of the number of steps taken by a simple random walk on $\mathbb Z$, started at $0$, before it first hits $-1$.  (In particular, $\mathbb E[K] = \infty$.)
\end{enumerate}

We denote by $V_i$ the symbol corresponding to the $i$th record minimum of $\cou_n$, counting forward from zero, if $i$ is positive, and the $-i$th record minimum of $\cou_n$, counting backward from zero, if $i$ is negative (see Figure \ref{f.excursions}).  We will introduce notation to describe the locations of the $V_i$ in the proof of Lemma \ref{l.expEfinite}. 

Denote by $E_i$ the {\em reduced form of} the word strictly in between the symbols corresponding to $V_i$ and $V_{i-1}$ (if $i$ is positive) or between $V_i$ and $V_{i+1}$ (if $i$ is negative), where $V_0$ is for the moment formally taken to be the ``zero point''  located between the symbols  $X(0)$ and $X(1)$ (see Figure \ref{f.excursions}). We do not define $E_i$ for $i = 0$. Note that the $V_i$ are i.i.d.\ equal to $\H$ or $\C$ (each with probability $.5$) when $i < 0$ and equal to $\h$, $\ch$, or $\x$ (with probabilities $(1-p)/2$, $(1-p)/2$ and $p$) when $i>0$. The $E_i$ are i.i.d.\ excursion words, each with same law as the $E$ described above, and are independent of the $V_i$.  (The fact that $E_1$ and $E_{-1}$ are identically distributed follows from the fact that an excursion of the simple random walk $\cou_n$ is equivalent in law to its time reversal.  Indeed, once we condition on the trajectory of $\cou_n$ over an excursion, we can sample the corresponding symbols by tossing independent coins to replace each upward step with an $\C$ or $\H$ and each downward step with a $\ch$, $\h$, or $\x$.)

\begin {figure}[!ht]
\begin {center}
\includegraphics [height=1.3in]{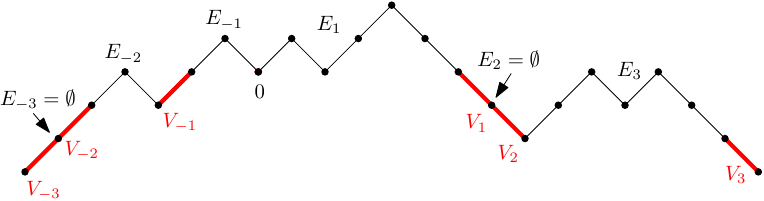}
\caption {\label{f.excursions} Possible graph of $\mathcal C_n$ as a function of $n$.  Up-going edges correspond to burgers, down-going edges to orders.  The $V_i$ correspond to edges that reach record minima when one starts from zero and moves in one (positive or negative) direction.  The $E_i$ correspond to the excursions in between the $V_i$.  (Recall a slight asymmetry in the notation: the first edge to the right of vertex labeled $0$ corresponds to $X(1)$ while the first edge to the left corresponds to $X(0)$.)}
\end {center}
\end {figure}


\begin{lemma} \label{l.expEfinite}
Suppose that $p$ is such that the $\chi$ defined in Section \ref{ss.variance} satisfies $\chi < 2$.  Then the expected word length $\mathbb E[|E|]$ is finite.  Moreover, the expected number of symbols in $E$ of each of the four types in $\{\C,\H,\ch,\h \}$ is $\mathbb E[|E|]/4$.
\end{lemma}

\begin{proof}
Since $E$ has the same number of symbols in $\{\C, \H \}$ as in $\{\ch, \h \}$ (recall that $E$ has no $\x$ symbols), the second statement is immediate from the first by symmetry.  To prove that $\mathbb E[|E|] < \infty$ it suffices to show that the expected number of burgers in the word $E_{-1}$ is finite.

To this end, let $J_1$ be the smallest non-negative integer $j$ for which $X(-j,0)$ has a $\C$ or $\H$ symbol.  This is the same as the $J$ defined in Section \ref{ss.variance} except that we count from $0$ instead of from $-1$.  (The law of the word $X(-J_1,0)$ is the same as the law of the word $X(-J,-1)$ in Section \ref{ss.variance}.  We count from $0$ here because $X(0)$ is the rightmost symbol involved in the definition of $E_{-1}$ and $V_{-1}$, see Figure \ref{f.excursions}.)  Then let $J_2$ be the smallest value greater than $J_1$ for which $X(-J_2, -J_1-1)$ has a $\C$ or $\H$ symbol, and so forth.  The words $X(-J_{k+1}, -J_k-1)$ are i.i.d.\ and all have the same laws as $X(-J,-1)$ (and each a.s.\ has length at least one).  In particular, if $\chi < 2$, then we have by \eqref{e.diswordlength} that $\mathbb E[\cou(-J_{k+1}, -J_k-1)] = 2-\chi > 0$ for each $k$.

This implies that (if we write $J_0=-1$) the sum
\begin{equation} \label{eqn:thismartingale} \sum_{k=1}^m \cou(-J_k, -J_{k-1}-1) - (2-\chi)m \end{equation} is a martingale indexed by $m$. Note that the sum $\sum_{k=1}^m \cou(-J_k, -J_{k-1}-1)$ can increase by at most $1$ at each increment. Let $M$ be the smallest $m$ for which $\sum_{k=1}^m \cou( -J_k, -J_{k-1}-1) = 1$. The fact that $\sum_{k=1}^m \cou( -J_k, -J_{k-1}-1) = 1$ is a sum of i.i.d.\ integer-valued random variables, each with expectation greater than zero, together with the law of large numbers implies that the sum tends to $\infty$ as $m \to \infty$. Since each of the random variables in the sum is at most $1$ we conclude that $M < \infty$ a.s. Since the martingale \eqref{eqn:thismartingale} is bounded above up to time $M$, the optional stopping theorem for martingales bounded above implies that the expected value of \eqref{eqn:thismartingale} at this stopping time is at least $0$, which implies that $\mathbb E[1- (2-\chi)M] \geq 0$ and hence $\mathbb E[M] \leq 1/(2-\chi)$.  Since $\cou(-J_M,0) = 1$, it follows that $-J_M$ is left of the location of $V_{-1}$ (since the latter corresponds to the {\em first} time the count reaches $1$, counting left from zero).  Since the number of burgers in $X(-j,0)$ is increasing in $j$, it follows that $X(-J_M, 0)$ has at least as many burgers as $\overline{V_{-1}E_{-1}}$.  The expectation of the number of burgers in $X(-J_M,0)$ is finite (since $X(-J_M, 0)$ has at most $M$ burgers), and hence so is the expectation of the latter.
\end{proof}

The lemmas that follow will involve several sequences related to the $J_m$ defined above (each defined for $m \geq 1$).  For convenience we define them here:

\begin{enumerate}
\item {\bf $m$th empty order stack:} $O_m$ is the $m$th smallest value of $j \geq 0$ with the property that $X(-j,0)$ has an empty order stack. We have not proved that $O_1$ is a.s.\ finite, but if it is then we may observe that the words $X(-O_m, -O_{m-1}-1)$ are i.i.d.\ (if we formally set $O_0 = -1$) and that each has an empty order stack.
\item {\bf $m$th empty burger stack:} $B_m$ is the $m$th smallest value of $j \geq 1$ with the property that $X(1,j)$ has an empty burger stack.
   We have not proved that $B_1$ is a.s.\ finite, but if it is, then we may observe that the words $X(B_{m-1}+1, B_m)$ are i.i.d.\ (if we formally set $B_0 = 0$) and that each has an empty burger stack.
\item {\bf $m$th left record minimum:} $L_m$ is the smallest value of $j \geq 0$ for which $\cou(-j,0) = m$.  Thus $X(-L_m, 0) = \overline{V_{-m}E_{-m} \ldots V_{-1} E_{-1}}$ and $X(-L_m, -L_{m-1}-1) = V_{-m} E_{-m}$ (if we formally set $L_0 = -1$).
\item {\bf $m$th right record minimum:} $R_m$ is the smallest value of $j \geq 1$ for which $\cou(1,j) = -m$.  Thus $X(1, R_m) = \overline{E_1V_1 \ldots E_m V_m}$ and $X(R_{m-1}+1, R_m) = E_mV_m$ (if we formally set $R_0 = -1$).
\item {\bf $m$th hamburger-order-free left minimum:} $L^H_m$ is the $m$th smallest value of $j \geq 0$ with the property that $X(-L^H_j, 0)$ has no hamburger orders and $L^H_j =L_{j'}$ for some $j'$ (i.e., $-L^H_j$ corresponds to a new record minimum of $\cou(\cdot, 0)$).  As in the cases above, the words $X(-L^H_j, -L^H_{j-1} -1)$ are i.i.d., and each one is a product of some sequence of words of the form $X(-L_{j'}, -L_{j'-1}-1)$.
\item {\bf $m$th hamburger-free right minimum:} similarly, $R^H_m$ is the $m$th smallest value of $j \geq 1$ with the property that $X(1,R^H_m)$ has no hamburgers and $R^H_j =R_{j'}$ for some $j'$.
\end{enumerate}

\begin{lemma} \label{l.finiteexpectationequivalence}
The following are equivalent:
\begin{enumerate}
\item \label{i.eleqinfty} $\mathbb E[|E|] < \infty$.
\item \label{i.hamklleqinfty} $\mathbb E[|X(-L^H_1,0)|] < \infty$.
\item \label{i.hamkrleqinfty} $\mathbb E[|X(1,R^H_1)|] < \infty$.
\item\label{i.klleqinfty} $O_1$ is a.s.\ finite and $\mathbb E[|X(-O_1,0)|] < \infty$.
\item \label{i.krleqinfty} $B_1$ is a.s.\ finite and $\mathbb E[|X(1,B_1)|] < \infty$.
\end{enumerate}
\end{lemma}

\begin{proof}
{\bf \ref{i.eleqinfty} implies \ref{i.hamklleqinfty}:}    Assume \ref{i.eleqinfty}.  Recall that the words $E_{-j}$ are i.i.d., that each a.s.\ has no $\x$ symbols and as many orders as burgers (possible example: $\ch \h \h \ch \C \C \C \H$), that the expected number of symbols of each of the types $\{\ch, \h, \C, \H \}$ in $E_{-1}$ is finite, and that these expectations are equal.  Recall also that the $V_j$ (for $j < 0$) are independent i.i.d.\ samples from $\{\C, \H \}$.
Let $H(m)$ be the number of hamburger orders in $X(-L_m,0) = \overline{V_{-m}E_{-m} \ldots V_{-1} E_{-1}}$.
Note that for any $m > 1$ we have that \begin{equation} \label{e.deltahm} H(m) = \max \{H(m-1) - h_m, 0 \} + o_m, \end{equation} where $h_m$ is the number of hamburgers in $\overline{V_{-m} E_{-m}}$ and $o_m$ is the number of hamburger orders in $\overline{V_{-m} E_{-m}}$.  Note that the pair $(h_m, o_m)$ is independent of $H(m-1)$. Note also that $\mathbb E[h_m] = \mathbb E[o_m] + .5$ (since the expected number of hamburger orders in $E$ equals the expected number of hamburgers and $V_{-m} = \H$ with probability $.5$).

Now we can rewrite \eqref{e.deltahm}, using the fact that $\max\{A,0 \} = A - 1_{A<0} A$, as $$H(m) - H(m-1) = o_m - h_m + (h_m - H(m-1)) 1_{\{H(m-1)-h_m<0 \}},$$ and
$\mathbb E[(h_m-k) 1_{h_m > k} ] \leq \mathbb E[h_m 1_{h_m > k }] \downarrow 0$ as $k \to \infty$ by assumption.
Thus, there exists a $C > 0$ such that for all $k > C$ we have $\mathbb E[H(m)-H(m-1) | H(m-1) = k] \leq - .4$.  It follows from Lemma \ref{l.finiteexpectationtozero}, applied with $\mathcal Z_m = H(m)$ and $Z_m = \overline{V_{-m}E_{-m}}$, that the expected number of $V_{-j}E_{-j}$ sequences concatenated to produce $X(-L^H_1,0)$ is finite, and it then follows from \ref{i.eleqinfty} and Lemma \ref{l.finitelystoppedexpectedsum} that the expected sum of the lengths of these words is finite.

{\bf \ref{i.eleqinfty} implies \ref{i.hamkrleqinfty}:}
The argument used to show that \ref{i.eleqinfty} implies \ref{i.hamklleqinfty} applies almost verbatim here, using $R_m$ in place of $L_m$.

{\bf \ref{i.hamklleqinfty} implies \ref{i.klleqinfty}:}
We use the argument used to show \ref{i.eleqinfty} implies \ref{i.hamklleqinfty}.  In that case, we concatenated i.i.d.\ words of the form $X(-L_m, -L_{m-1} -1) = \overline{V_{-m} E_{-m}}$ until we produced a word with no hamburger orders, which was equal to $X(-L^H_1, 0)$.  In this case, we start by concatenating the i.i.d.\ words of the form $X(-L^H_m, -L^H_{m-1} - 1)$, each of which has no orders of type $\h$ or $\x$ (possible example: $\ch \ch \ch \ch \C \H \C \C \C \H$), and we continue until we produce a word with no orders at all.  This word is necessarily at least as long as $X(-O_1,0)$ (since the burger stack is increasing in time, and $X(-O_1,0)$ corresponds to the {\em first} time the order stack is empty), so it suffices to bound its expected length.

We are assuming that the expected length of $X(-L^H_1, 0)$ is finite; since this word has at least as many burgers as $X(-L_1,0)$, this implies that the expected number of burgers in $X(-L_1,0)$ is also finite, and since the number of orders in $X(-L_1,0)$ is one less than the number of burgers, we find that the expected total length of $X(-L_1,0)$ is finite, as is the expected total length of $E_{-1}$.  That is, we have established that \ref{i.hamklleqinfty} implies \ref{i.eleqinfty}. We may therefore use the fact (established at the end of the \ref{i.eleqinfty} implies \ref{i.hamklleqinfty} argument) that the expected number of $V_{-j}E_{-j}$ sequences concatenated to produce $X(-L^H_1,0)$ is some finite number $k$.

Note that the expected net number of cheeseburgers in $X(-L_m,-L_{m-1}-1) = V_{-m} E_{-m}$ (i.e., the expected number of $\C$ symbols minus the number of $\ch$ symbols) is $1/2$, since this net number is additive when we concatenate words with no $\x$ symbols.  We similarly conclude by  \eqref{eqn::stoppingtotal} that the expected net number of cheeseburgers in $X(-L^H_1,0)$ is $k/2$, where $k$ is as defined just above.

As we concatenate i.i.d.\ copies of the word $X(-L^H_1,0)$, the number of $\ch$ symbols in the reduced form of the concatenated word is a Markov chain on $\Z_+$. The above observations imply that the expected change in this Markov chain value, during a step that starts at position $j$, is a value that tends to $-k/2$ as $j \to \infty$. For any $j$, this expected change is bounded above by the finite expected length of $X(-L^H_1,0)$. Thus, Lemma \ref{l.finiteexpectationtozero}, implies that the expected number of steps until this chain reaches zero is finite.  Similarly, Lemma \ref{l.finitelystoppedexpectedsum} then implies that the expected total length of the words concatenated is finite.

{\bf \ref{i.hamkrleqinfty} implies \ref{i.krleqinfty}:}
This is essentially the same as the proof that \ref{i.hamklleqinfty} implies \ref{i.klleqinfty}.

{\bf Either \ref{i.klleqinfty} or \ref{i.krleqinfty} implies \ref{i.eleqinfty}:}  Note that the number of burgers in $X(-O_1,0)$ is at least the number in $E_{-1}$, and the number of orders in $X(1,B_1)$ is at least the number in $E_1$.  Thus if either $\mathbb E[|X(-O_1,0)|] < \infty$ or $\mathbb E[|X(1,B_1)|] < \infty$, then (recalling that the expected number of orders in $E$ equals the expected number of burgers in $E$) we have $\mathbb E[|E|] < \infty$.
\end{proof}

\begin{lemma} \label{l.stackfraction}
If $\mathbb E[|E|] < \infty$ then the limit, as $n \to \infty$, of the fractions of $\H$ symbols among the top $n$ elements in $X(-\infty,0)$ is equal to $.5$ almost surely. Similarly, as $n$ tends to infinity, the fraction of $\ch$, $\h$ and $\x$ symbols among the leftmost $n$ elements of $X(1,\infty)$ tend almost surely to constants (the first two equal, by symmetry). These constants are all positive (except in case $p=1$, where one has only $\x$ symbols and no $\ch$ or $\h$ symbols).

On the other hand, if $\mathbb E[|E|] = \infty$ then the limit as $n \to \infty$ of the fraction of $\x$ symbols among the leftmost $n$ elements of $X(1,\infty)$ tends almost surely to zero.
\end{lemma}
\begin{proof}
Consider the sequence of words $X(-k,0)$, indexed by positive $k$.  As we have noted before, the words $X(-O_m,-O_{m-1}-1)$ are i.i.d., each consisting entirely of burgers, and $X(-\infty,0)$ is the concatenation of these all-burger words.  Lemma \ref{l.finiteexpectationequivalence} implies that if $\mathbb E[|E|] < \infty$ then the lengths of these burger sequences have finite expectation, and the law of large numbers then implies the number of hamburgers and cheeseburgers in $X(-O_m,0)$ are both (up to $o(m)$ errors) given by constant multiples of $m$, almost surely.  This implies the first statement in the lemma.  The proof of the second statement is analogous, using $B_m$ instead of $O_m$. To see that the constants are positive, it suffices to observe that $X(1,B_1)$ has {\em some} positive probability of containing an order of each of the three order types.

For the final statement, we note that an $\x$ symbol can be added to the order stack of the sequence $X(1,k)$ only at times when the burger stack is empty.  The number of $\x$ symbols in $X(1,B_m)$ can grow as a function of $m$, but it can grow by at most $1$ each time that $m$ increases by $1$.  If $\mathbb E[|E|] = \infty$, then Lemma \ref{l.finiteexpectationequivalence} and the law of large numbers imply that the number of orders in $X(1,B_m)$ a.s.\ grows faster than any constant times $m$, while the number of $\x$ symbols grows like a constant times $m$.
\end{proof}

The following proposition is not needed for the proof of Theorem \ref{t.brownianlimit}, but we include it because it will be interesting from the point of view of random planar maps.  (In a sense, it will imply that the infinite random surface models one obtains when $p > 1/2$ a.s.\ have infinitely many small bottlenecks surrounding any given point.)

\begin{proposition} \label{p.bottleneck}
Let $K$ be the infimum over the set of positive numbers such that there exist $m_- < 0$ and $m_+>0$ such that \begin{enumerate}
\item $X(m_-, 0)$ is a word with no orders and $K$ burgers.
\item $X(1,m_+)$ is a word with no burgers and $K$ orders.
\item $X(m_-, m_+) = \emptyset$
\end{enumerate}
If $\mathbb E[|E|] < \infty$ then $K$ is a.s.\ finite and $\mathbb E[K] < \infty$.
\end{proposition}
\begin{proof}

Assume $\mathbb E[|E|] < \infty$.  As noted in the proof of Lemma \ref{l.stackfraction}, the gaps between $a_m := |X(-O_m,0)|$ (as $m$ increases) and between $b_m := |X(1,B_m)|$ (as $m$ increases) are i.i.d.\ positive random variables with finite expectation (and clearly every possible positive gap length has positive probability). 

Write $\alpha_k = k - \max \{a_m : a_m \leq k \}$ and $\beta_k = k - \max\{b_m: b_m \leq k \}$. Both $\alpha_k$ and $\beta_k$ can be understood as ergodic Markov chains. To see this, note that if $\alpha_k =j$, then $\alpha_{k+1} \in \{j+1, 0 \}$ and $\mathbb P\bigl [\alpha_{k+1} = 0 \,|\, \alpha_k = j\bigr] =\mathbb P\bigl[a_1 = j+1 \,| \, a_1 > j\bigr]$. The transition kernel for $\beta_k$ can be written similarly. Since $a_\cdot$ and $b_\cdot$ each have increments with finite expectation, the corresponding ergodic Markov chains assume the value $0$ (at stationarity) with probability one over the corresponding expectation. Since $\alpha_\cdot$ and $\beta_\cdot$ are independent, the pair $(\alpha_\cdot, \beta_\cdot)$ is also a Markov chain that (at stationarity) assumes the value $(0,0)$ with positive probability.
 
Call a non-negative integer $M$ a ``match up length'' if there exists an $i$ such that $a_i = |X(-O_i, 0)| = M$ and a $j$ such that $b_j = |X(1,B_j)| = M$.  The match up lengths are a random subset of $\mathbb Z_+$ and correspond to the $k$ for which $(\alpha_k, \beta_k) = (0,0)$. The discussion above implies that the gaps between successive match up lengths are i.i.d.\ random variables with finite expectation.  Denote by $M_k$ the $k$th such match up length.

For each $M_k$, we have a word $X(-O_i,0)$ of length $M_k$ comprised entirely of burgers and a word $X(1,B_j)$ of length $M_k$, comprised entirely of orders.  We claim that $X(-O_i,B_j)$ is a word that consists of a sequence of orders of one type (either $\ch$ or $\h$) followed by a same-length sequence of burgers of opposite type (either $\C$ or $\H$), e.g.,
$$\ch\ch\ch\ch\ch\H\H\H\H\H.$$
(This can be seen by starting with the all-burger word $X(-O_i,0)$ and multiplying on the right by the symbols in $X(1,B_j)$ one at a time.  The only way there can be orders in $X(-O_i,B_j)$ is if all of the burgers of one type are consumed before this process terminates, and in that case only orders corresponding to that burger type can appear in $X(-O_i,B_j)$.)
When we shift from $k$ to $k+1$, we multiply this word on the left and right by random words $W_1$ and $W_2$.  We know that $\cou(W_1) + \cou(W_2) = 0$.  By symmetry, the expected number of symbols of type $\H$ in this pair of words equals the expected number of symbols of type $\C$.  (These expectations are finite by Lemma \ref{l.finiteexpectationequivalence}.) Similarly for $\h$ and $\ch$ symbols.  The expected number of $\x$ symbols is some positive constant $a$.

If the length of $X(-O_i,B_j)$ is long enough, then the probability that either $W_1$ or $W_2$ is longer than half of that length is close to zero.  Thus, conditioned on the length being at least some constant value, the expected change in length as one goes from $k$ to $k+1$ is close to $-a/2$.  It then follows from Lemma \ref{l.finiteexpectationtozero} that the expected value of the smallest $k$ for which this length is zero is finite, and the result then follows from Lemma \ref{l.finitelystoppedexpectedsum}.
\end{proof}

\subsection{Some monotonicity observations} \label{ss.monotonicity}


In this section, it will be convenient to fix a semi-infinite stack present at time zero (the stack $X(-\infty, 0)$) and also to relax the assumption in Definition \ref{andef} that $(\dis_n, \cou_n) = (0,0)$.  To give some notation for this, let $S_0$ be a semi-infinite stack of burgers containing infinitely many burgers of each type {\em together} with a corresponding diagram in $\R^2$ such as the one given in Figure \ref{stackchange}.  The diagram is determined by the burger stack once we know the location of the uppermost vertex (the {\em tip}) in Figure \ref{stackchange}, which we now allow to be {\em any} lattice point of $\mathbb Z^2$ whose coordinate sum is even.  Slightly abusing notation, we will write $(\dis_0,\cou_0)$ for the location of the tip of $S_0$ (which, in this section only, we do not require to be at the origin).

Given $S_0$, we can generate a sequence of stacks $S_1, S_2, \ldots$ in the usual manner by applying the moves in Figure \ref{stackchange} that correspond to $X(1), X(2), \ldots$.  For $n > 0$ we write $(\dis_n, \cou_n) := \bigl(\dis_0 + \dis(1,n), \cou_0 + \cou(1,n)\bigr)$ as before, so that $(\dis_n, \cou_n)$ is the location of the tip of $S_n$.

Given any such embedded stacks $S$ and $\tilde S$ we write $S \leq \tilde S$ if the tips of the path lie on the same horizontal line and the path describing $S$ (as in Figure \ref{stackchange}) lies to the left of the path describing $\tilde S$ --- i.e., every horizontal line intersects the $S$ path at a point equal to or left of where it intersects the $\tilde S$ path, see Figure \ref{twostacks}.  The following lemma is immediate from an inspection of the cases in Figure \ref{stackchange}.

\begin {figure}[!ht]
\begin {center}
\includegraphics [height=1in]{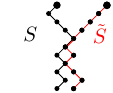}
\caption {\label{twostacks} Upper portion of embedded stacks $S \leq \tilde S$ (partially overlapping).}
\end {center}
\end {figure}

\begin{lemma} \label{l.monotonicity}
Suppose embedded stacks $S_0 \leq \tilde S_0$ are fixed and that $S_1, S_2, \ldots$ and $\tilde S_1, \tilde S_2, \ldots$ are generated from $S_0$ and $\tilde S_0$ respectively using the same sequence $X(1), X(2), \ldots$ as in Figure \ref{stackchange}.  Then $S_n \leq \tilde S_n$ for all $n > 0$.
\end{lemma}

We can use this result to deduce the following:

\begin{lemma} \label{l.condexpbound}
Fix $N > 0$ and $S_0$ and let $X(1), X(2), \ldots$ be chosen from $\mu$.  Then $\mathbb E[\dis_N | X(j) :j \leq k ]\}$ is a martingale in $k$ with increments of magnitude at most $2$ (which obtains the value $\dis_N$ at $k=N$).
A similar result holds when one further conditions on the value of the sequence $\cou_n$ for $n \in \{1,2,\ldots,N\}$: namely, $\mathbb E[\dis_N | X(j): j \leq k, \cou_j: 0 \leq j \leq N]$ is a martingale in $k$ with increments of magnitude at most $2$.
\end{lemma}

The lemma implies, in particular, that if we {\em fix} any choice of $S_0$ and fix $N > 0$ and then sample the sequence $X(1), X(2), \ldots$ chosen from $\mu$, the variance of $\dis_N$ is at most $4N$.  (The conditional variance of $\cou_N$ is exactly $N$.)

\begin{proof}
We claim that changing the value $X(j)$ for a single $j \in \{1,2,\ldots,N \}$ (while leaving the other values fixed) always changes $\dis_N$ by at most $2$.  Both statements in the lemma are immediate consequences of this.  We will establish the claim using Lemma \ref{l.monotonicity}.

First, if we change $X(j)$ in a way that does not affect $\cou_j$ (i.e., we change $X(j)$ from one burger type to another, or from one order type to another) then it is clear from Figure \ref{stackchange} that the modified $S_j$ obtained lies between the original $S_j$ and the original translated by $2$ units (either left or right).  It follows from Lemma \ref{l.monotonicity} that the change to $X(j)$ alters the tip of the final stack $S_N$ by at most $2$ units.  If we change $X(j)$ in a way that {\em does} affect $\cou_j$ then a variation of this argument still applies.  Suppose the modified $S_j$ has a tip two units higher vertically than the original.  Then if we translate the modified $S_j$ {\em down} by two units, it will lie within two units to the left or right of the original $S_j$ (again, by inspection of Figure \ref{stackchange}), and one can then apply Lemma \ref{l.monotonicity} as before to show that this remains true after both stacks evolve for $N-j$ steps using the same $X(j+1) \ldots X(n)$ sequence.
\end{proof}


The following is a fairly simple and standard observation about the tail behavior of martingales with small increments.  It follows e.g.\ from Lemma 2 of \cite{azuma}. (We include a short proof here for completeness, but this can be skipped.)

\begin{lemma} \label{l.martingalebound}
Let $\beta_j$ be a martingale in $j$ whose increments have magnitude at most $1$, with $\beta_0 = 0$.  Then for each real $a>0$ and integer $n > 0$ we have $$P \bigl(\max_{j \in \{1,2,\ldots,n \}}n^{-1/2} \beta_j \geq a \bigr) \leq e^{-a+1/2}.$$
\end{lemma}
\begin{proof}
First, we observe that $\frac{e^b + e^{-b}}{2} \leq e^{b^2/2}$ for all $b$.  (Indeed, by Taylor expansion, the left side is $1 + b^2/2! + b^4/4! + b^6/6! + \ldots$ while the right side is $1 + b^2/(2^1 1!) + b^4/(2^2 2!) + b^6/(2^3 3!) + \ldots$.)  Thus, if the increments of $\beta_j$ were exactly $\pm1$ (each with probability $1/2$), with $\beta_0 = 1$, then we would have $\mathbb E[e^{b \beta_1}] \leq e^{b^2/2}$ and $\mathbb E[e^{b \beta_1 - b^2/2}] \leq 1$.  We claim that the same is true if we allow for $\beta_1 \not \in \{-1,1 \}$ and insist only that $|\beta_1| \leq 1$ a.s.\ and $\mathbb E \beta_1 = 0$.  (Indeed, one may first choose an increment $\beta_1$, and then, given this, choose a ``modified increment'' $\beta_1' \in \{-1,1\}$ using a biased coin whose probability is chosen so that conditional expectation of the modified increment is $\beta_1$.  Jensen's inequality implies $\mathbb E[e^{b \beta_1}] \leq \mathbb E[e^{b \beta_1'}]$.)  The argument above shows more generally that $e^{b \beta_j - j b^2/2}$ is a supermartingale indexed by $j$.

Taking $b = n^{-1/2}$ produces a supermartingale in $j$ whose value at time $j=0$ is $1$ and whose value at general $j$ time is $e^{n^{-1/2} \beta_j - \frac{j}{2n}}$.  The expectation of the latter quantity is at most $1$, by the optional stopping theorem, so the probability that it ever reaches $e^{a-1/2}$ is at most $e^{-a+1/2}$.  To conclude, we note that $n^{-1/2} \beta_j \geq a$ implies $e^{n^{-1/2} \beta_j - \frac{j}{2n}} \geq e^{a-1/2}$ for $j \in \{1,2,\ldots,n \}$.
\end{proof}

\begin{lemma} \label{l.translateprob}
Fix any $p \in [0,1]$.  Fix the time zero stack $S_0$.  Then there are positive constants $C_1$ and $C_2$ such that, for any choice of $S_0$, $a>0$, and $n \geq 1$, the probability that $|\dis_j|>a \sqrt n$ for some $j \in \{0,1,2,\ldots,n \}$ is at most $C_1 e^{- C_2 a}$.
\end{lemma}

\begin{proof}
Without loss of generality, assume $\cou_0 = \dis_0 = 0$.  To sample $X(1), X(2), \ldots$ we may first sample the sequence $\cou_1, \cou_2, \ldots$ and then conditioned on that choose the types of the burgers and orders (which in turn determine the $\dis_j$ sequence).  Write $M = a \sqrt{n} / 8$ and note that
\begin{equation}\label{e.dismaxbound} \mathbb P \bigl\{ \max_{j \in \{1,2,\ldots,n \}} |\cou_j| >  M \bigr \} \leq 2 e^{-a/8 +1/2}\end{equation} by Lemma \ref{l.martingalebound}.  Thus, we may restrict attention to the event $\mathcal E_n$ that this does not occur, in which case we are conditioning on values of $\cou_j$ that satisfy $|\cou_j| \leq M$ for $j \in \{1,2,\ldots,n\}$.

Using Lemma \ref{l.condexpbound} we have that
\begin{equation}\label{eqn::mj} m_j := \mathbb E[\dis_n | X(i) :i \leq j, \cou_i : 0 \leq i \leq n ]\end{equation} is a martingale in $j$ for $j \in \{0,1,\ldots,n\}$ with increments of size at most $2$. The value of $\dis_j$ depends on the stack at time $j$ and does not depend on any of the edges in $S_0$ that lie below $-M$.  This is because if one considers two {\em different} choices of $S_0$ which agree above level $-M$, and then takes the {\em same} choice of symbols $X(0), X(1), X(2),\ldots$ whose values imply $|\cou_j| \leq M$ for $j \in \{1,2,\ldots,n\}$, then by induction the updated stack $S_j$ will look the same above level $-M$ for each $j$ between $0$ and $n$, and since all $\x$ symbols get matched to burgers above level $-M$, we may conclude that (on the event $\mathcal E_n$) $\dis_j$ does not depend on the choice of $S_0$ below level $-M$.  We can also rewrite \eqref{eqn::mj} as $m_j = \mathbb E [\dis_n | S_j , \cou_i: 0 \leq i \leq n]$.

Now {\em fix} some value $j \in \{1,2,\ldots,n\}$ and let $\tilde S_j$ be the stack obtained by swapping the burgers in $S_j$ above level $-M$ {\em and} subsequently translating the stack by $4M$ units to the left (i.e., translating the stack by the amount that ensures that the corresponding $\dis_j$ value is translated by $4M$, but $\cou_j$ remains the same). Then we see that $\tilde S_j$ lies strictly to the left of $S_j$. One can then see inductively (by Lemma \ref{l.monotonicity}) that this will continue to be the case if we evolve both stacks in parallel by adding the same symbols $X_{j+1}, X_{j+2}, \ldots$ to each. Let $\tilde \dis_j$ be the discrepancy of $\tilde S_j$, so that (by definition) we have $\tilde \dis_j = \dis_j - 4M$. Let $\tilde \dis_i$ for $j \leq i \leq n$ be obtained by letting the stack
evolve by the addition of the symbols $X_j+1, X_j+2, \ldots, X_n$.  Set $\tilde m_j = \E[\tilde \dis_n
| \tilde S_j ; \cou_i
, 0 \leq i \leq n)],$
which is also $\mathbb E[\tilde \dis_n| S_j,  \cou_i: 0 \leq i \leq n)]$.  Now observe by symmetry that on the event $\mathcal E_n$ we have \begin{equation} \label{eqn::disjmj} \tilde \dis_j - \tilde m_j = -(D_j - m_j). \end{equation} By the fact (noted above) that $\tilde S_i$ is left of $S_i$ for $j \leq i \leq n$ one has $\tilde \dis_n \leq \dis_n$ and therefore $\tilde m_j \leq m_j$, and we also recall that $\tilde \dis_j = \dis_j - 4M$. Rewriting the LHS of \eqref{eqn::disjmj} using these two facts, we get
$$\dis_j - 4M -m_j \leq -\dis_j + m_j,$$
which implies $\dis_j - m_j \leq 2M$ on $\mathcal E_n$.

An analogous statement holds if we define $\tilde S_j$ using translation to the right instead of to the left. Thus, on the event $\mathcal E_n$, we have $|\dis_j - m_j| \leq 2M = \frac{a}{4} \sqrt{n} $  for all $j \in \{0,1,\ldots,n\}$. Thus, in order to have $|\dis_j - \dis_0| \geq a \sqrt{n}$ for some $j \in \{1,2,\ldots,n\}$ we must have $|m_j-m_0| \geq a \sqrt{n}/2$.  Applying Lemma \ref{l.martingalebound} to $m_j/2$, we find that the conditional probability (given $\mathcal E_n$) that $|m_j - m_0|$ exceeds $\sqrt{n} a/2$ on the interval $\{1,2,\ldots,n\}$ is at most $e^{-a/4 + 1/2}$.  This combined with \eqref{e.dismaxbound} gives the lemma.
\end{proof}

\begin{lemma} \label{l.wordlength}
The length of $X(1,n)$ is typically of order $\sqrt{n}$ or smaller regardless of $p$.  More precisely, there are positive constants $C_1$ and $C_2$ (independent of $n$ and $a>0$) such that
$$P(\sup_{i\in \{1,2, \ldots, n\}} |X(1,i)| > a \sqrt n) \leq C_1 e^{-C_2 n}.$$\end{lemma}

\begin{proof}
Fix an initial stack $S_0$ to be alternating between $\C$ and $\H$.  Then observe that {\em if} $\dis_k$ and $\cou_k$ both fluctuate by at most $a \sqrt{n}/5$, as $k$ ranges from $1$ to $n$, then no burger on the stack $S_0$ strictly below height $-2 a \sqrt{n}/5$ will have been consumed during the first $n$ steps.  (If the first burger below that height to be consumed were consumed at step $j$ then all of the burgers above it in the stack at step $j$ --- let $k$ be the number of such burgers --- would have to be of the same type. This would require that either $k < a \sqrt{n}/5$, in which case $|\cou_j|$ would have to exceed the assumed bound, or that $k \geq a \sqrt{n}/5$, in which case $|\dis_j|$ would have to exceed the assumed bound, since all the burgers below the top $k$ must strictly alternate between $\C$ and $\H$.)  Thus, on this event the total number of orders in $X(1,i)$ is less than $2a \sqrt{n}/5$ for each $i \in \{1,2,\ldots,n\}$.  Since $\cou_k$ fluctuates by at most $a \sqrt{n}/5$ this also implies that the total number of burgers in $X(1,i)$ is at most $3 a \sqrt{n}/5$  for each $i \in\{1,2,\ldots,n\}$, and hence $|X(1,i)| < a \sqrt{n}$ for each $i \in\{1,2,\ldots,n\}$.  The result thus follows from Lemma \ref{l.translateprob} and the analogous bound that applies when $\dis$ is replaced by $\cou$.
\end{proof}

\subsection{Proof of Lemma \ref{l.chigives}} \label{ss.phase}

Lemma \ref{l.chigives} should not seem very surprising in light of the results we have established so far.   Lemma \ref{l.expEfinite} and Lemma \ref{l.stackfraction} show that when the Lemma \ref{l.chigives} hypothesis holds (i.e., $\chi \not = 2$) the stack $X(-\infty,0)$, embedded in the manner of Figure \ref{stackchange}, scales to a vertical line a.s.\ as one zooms out.  It is thus natural to expect that the left-right fluctuation of the time-evolution of the stack is small compared to the up-down fluctuation.  We divide Lemma \ref{l.chigives} into Lemmas \ref{l.chinot2zerovariance} and \ref{l.distightness}, which we state and prove below.

\begin{lemma} \label{l.chinot2zerovariance}
If $\chi \not = 2$ then $\Var[\dis(n)] = o(n)$.
\end{lemma}
\begin{proof}
First, we claim that if $\chi \not = 2$ then $n^{-1/2} \dis_n$ tends to zero in probability.  To see this, recall that the stacks $X(-\infty,0)$ and $X(-\infty, n)$ agree in law, and that the collection of the top $a \sqrt{n}$ burgers (for any fixed $a$) is likely to contain a roughly even distribution of hamburgers and cheeseburgers, in the sense of Lemma \ref{l.stackfraction}.  By Lemma \ref{l.wordlength}, the probability that the two stacks agree except for the top $a \sqrt{n}$ burgers (i.e., that one can find $j,k \leq a \sqrt{n}$ such that $X(-\infty,0)$ with the top $j$ burgers removed is the same as $X(-\infty,n)$ with the top $k$ burgers removed) is a quantity that remains bounded below as $n$ tends to infinity, and the bound can be made arbitrarily close to $1$ by making $a$ large enough.  It follows from this that the random variables $n^{-1/2}\dis_n$ tend to zero in probability.  We still need to show that the variances $\mathbb E[n^{-1} \dis_n^2]$ tend to zero.  This follows from the fact that the random variables $n^{-1} \dis_n^2$ tend to zero in probability, together with the uniform bounds on their tails given by Lemma \ref{l.translateprob}.
\end{proof}

\begin{lemma} \label{l.zerovariancemeansmaximumdnsmall}
If $\Var[\dis(n)] = o(n)$ then $n^{-1/2} \max \{|\dis_j| : 1 \leq j < nt \rfloor \}$ converges to zero in probability for each fixed $t > 0$.
\end{lemma}
\begin{proof}
The variance assumption immediately implies that $n^{-1/2} \dis_{\lfloor nt \rfloor }$ converges to zero in probability for each fixed $t$.  It also implies that the joint law of $n^{-1/2} \dis_{\lfloor nt \rfloor }$ at any fixed collection of $t$ values tends to zero in probability.

Now, suppose that we divide the first $t$ units of time into equal increments of length $\delta$ (for some small $\delta$). Then on each such increment, we would guess that with high probability the fluctuation of $n^{-1/2} \dis_{\lfloor tn \rfloor}$, as $t$ ranges through an interval of length $\delta$, should be very small when $n$ is sufficiently large and $\delta$ is small and fixed.  Indeed, the bound in Lemma \ref{l.translateprob} (applied with $a = \delta^{1/6}$) implies that the probability that there is even a single interval on which the fluctuation magnitude is greater than $a \delta^{1/2} = \delta^{2/3}$ will remain bounded above by some constant as $n \to \infty$, and this constant can be made arbitrarily small by taking $\delta$ sufficiently small.  Since we can take $\delta$ as close to zero as we like, the lemma follows.
\end{proof}

\begin{lemma} \label{l.distightness}
If $\chi \not = 2$ and $\Var[\dis(n)] = o(n)$ then $$\lim_{n \to \infty} \mathbb E\bigl[ |\dis(-n,-1)| 1_{J > n }\bigr]=0.$$
\end{lemma}
\begin{proof}
Let us assume $\Var[\dis(n)] = o(n)$ and proceed to derive the conclusion of the lemma.  Let $I$ be the smallest value of $j \geq 0$ for which $\cou(-j,-1) = 1$.  For each $n$, let $\mu_n$ be the measure whose Radon-Nikodym derivative with respect to $\mu$ is given by the (a.s.\ non-negative) quantity $\bigl(1-\cou(-n,-1)\bigr) 1_{I \geq n}$. Since $\bigl(1-\cou(-n,-1)\bigr)$ is a martingale (which reaches zero at time $I$) the optional stopping theorem implies that the expectation of $\bigl(1-\cou(-n,-1)\bigr) 1_{I \geq n}$ is $1$, so that $\mu_n$ is in fact a probability measure.
Informally, $\mu_n$ is the measure obtained from $\mu$ by conditioning on $a_j := \bigl( 1-\cou(-j,-1)\bigr)$ being positive until time $n$ and multiplying the probability of each $X(-n) \ldots X(-1)$ sequence by a quantity proportional to $a_n$.

Note that if we just consider the law $\mu$ conditioned on having $a_j$ eventually hit some $C > n$ before hitting zero, then the law of $\{a_1, a_2, \ldots, a_n \}$ under this conditional law agrees with $\mu_n$.  In both measures, if we are given the values of $a_i$ up to some $j \in \{0,1,\ldots,n-1\}$, then the conditional probabilities that $a_j$ goes up to $a_j+1$ and down to $a_j-1$ are proportional, respectively, to $a_j+1$ and $a_j - 1$.  Thus, there exists a measure $\mu_\infty$ whose law restricted to $X(j): j \leq n$ agrees with $\mu_n$, for each $n$.  One can sample from $\mu_\infty$ by first sampling the $a_n$ for all $n$ (an ordinary simple random walk for negative $n$, a walk ``conditioned to stay positive for all time'' for positive $n$), which determines $\cou(-n,-1)$, and then conditioned on that, choosing the burger and order types for each step independently from the usual conditional laws.  It is well known that the $\mu_\infty$ law of the process $n^{-1/2} a_{\lfloor t n \rfloor}$ (where $n$ is fixed, $t$ is a parameter) converges to that of a three-dimensional Bessel process as $n$ tends to infinity (which can be understood as ``Brownian motion conditioned to stay positive for all time'').  In particular, this implies that for every $\delta>0$ there exists a $b = b(\delta) > 0$ such that \begin{equation} \label{e.muncoudelta} \mu_n \{ b^{-1} < n^{-1/2} |\cou(-n,-1)| < b \} \geq 1-\delta \end{equation} for all sufficiently large $n$.

Using the measure $\mu_\infty$ allows us to convert an expectation into a probability: by definition, since $J \leq I$, \begin{equation} \label{e.couvalue} \mathbb E_\mu\bigl[ |1-\cou(-n,-1)|1_{J > n} \bigr]= \mu_n \{J > n\} = \mu_\infty \{J > n \} \end{equation}  We remark that by \eqref{e.diswordlength}, \eqref{e.chidef}, \eqref{e.chi2} and our assumption that $\chi \not = 2$, together with the fact that the quantity $\cou(-n,-1) 1_{J>n}$ that appears in \eqref{e.chi2} is non-positive, the values $\mu_\infty \{J > n \}$ converge to a positive constant as $n \to \infty$.  In other words, there is a positive $\mu_\infty$ probability that $J = \infty$, even though $J$ is $\mu$ a.s.\ finite.  An analog to \eqref{e.couvalue} is the following:
\begin{equation} \label{e.disvalue} \mathbb E_\mu\bigl[ |\dis(-n,-1)| 1_{J>n} \bigr]  = \mathbb E_{\mu_\infty} \bigl[\frac{|\dis(-n,-1)|}{1-\cou(-n,-1)} 1_{J > n }\bigr]. \end{equation}
The quantity on the right is an expectation of a value that is bounded above.  Hence, to show that it tends to zero, it suffices to show that the quantity $\frac{|\dis(-n,-1)|}{1-\cou(-n,-1)}$ tends to zero in probability under the $\mu_\infty$ measure.

For this, in light of \eqref{e.muncoudelta}, it suffices to show that for every $b>0$ and $ b'>0$, we have
\begin{equation} \label{e.distightnesscriterion} \lim_{n \to \infty} \mu_\infty \{ b^{-1} < n^{-1/2} |\cou(-n,-1)| < b  \,\,\,\,\text{and}\,\,\,\, n^{-1/2} |\dis(-n,-1)| > b' \} = 0.\end{equation}

We now fix such a $b$ and $b'$ and proceed to prove \eqref{e.distightnesscriterion}.  First let us make a remark about the strategy.  We want to understand the measure $\mu_\infty$ (where we condition $a_j$ to be positive) but most of the results in this paper apply to $\mu$ (which does not have such conditioning).  The rough idea that helps us make the connection is that if we produce the simple walk $a_j$ using $\mu$ and then {\em recenter} at a place where $a_j$ is locally minimal, then the recentered process looks (locally) similar to a sample from $\mu_\infty$.

Precisely, we may sample from $\mu$ and let $\tilde j$ be the (smallest) value $j \in \{0,1,\ldots,2n \}$ where $a_j$ obtains its minimum.  With probability at least $.25$, we have $\tilde j \in \{0,1,\ldots,n/2 \}$ for even $n$.  Conditioned on $j$ and $a_{\tilde j}$, the law of $a_k$ for $k \geq j$ is just that of a simple random walk conditioned not to go below $a_j$ until time $2n$. We claim that the conditional $\mu$ law of $a_{j - \tilde j} - a_{\tilde j}$ for $j \in \{1,2,\ldots,n \}$ is similar to the $\mu_\infty$ law of $a_j$ for $j \in \{1,2,\ldots, n\}$ in the sense that the Radon-Nikodym derivative between the two remains bounded between positive constants as $n$ tends to infinity.

Indeed, we can describe this Radon-Nikodym derivative explicitly. We have two probability measures on sequences of the form $r_1, r_2, \ldots, r_n$. The Radon-Nikodym derivative is a function of $r_n$ alone; recall that we restricting attention to paths for which $r_n < 2 b n^{1/2}$.  In the $\mu_\infty$ case, the probability of each path is proportional to $r_n$, which is in turn proportional to the probability that a random walk started at $r_n$ reaches height $2b \sqrt{n}$ before reaching zero. In the case of the conditional $\mu$ law, the probability of a path is proportional to the probability that walk started at $r_n$ fails to hit zero until at least the end what was the original length $2n$ interval (i.e., at least $n/2$ more steps).  It is not hard to see that given that a particular one of these two events occurs (reaching height $2b\sqrt{n}$ before zero or staying positive for $n/2$ more steps) the conditional probability that the other occurs is at least some constant (which does not tend to zero as $n \to \infty$), and that this implies the claim.

In light of the above claim, \eqref{e.distightnesscriterion} follows from Lemma \ref{l.zerovariancemeansmaximumdnsmall}.  If the $\mu_\infty$ law of $n^{-1/2}\dis(-n,-1)$ (conditioned on $b^{-1} < n^{-1/2} |\cou(-n,-1)| < b$) failed to converge to zero, then the $\mu$ law of the maximum of $(2n)^{-1/2} |\dis(-j,-1)|$ obtained on $(0,2n)$ would have to also fail to converge to zero uniformly, contradicting Lemma \ref{l.zerovariancemeansmaximumdnsmall}.
%
\end{proof}

\subsection{Proof of Theorem \ref{t.brownianlimit}} \label{ss.theoremproof}

We now complete the proof of Theorem \ref{t.brownianlimit}.  The case $p \geq 1/2$ follows from Lemma \ref{l.pchivariance} and Lemma \ref{l.zerovariancemeansmaximumdnsmall}, so we may assume $p < 1/2$.  For each fixed value of $t$, the variance limits described in Lemma \ref{l.pchivariance} guarantee that the variance of $n^{-1/2} \dis_{\lfloor t n \rfloor }$ converges to $\alpha t$ as $n \to \infty$, where $\alpha = \max \{1-2p,0\}$ as in the statement of Theorem \ref{t.brownianlimit}.  This implies that, at least subsequentially, the random variables $n^{-1/2} A_{t n}$ converge in law to a limit as $n \to \infty$ for each fixed $t$.  The same is true of the joint law of $n^{-1/2} A_{\lfloor t n \rfloor}$, where $t$ ranges over a finite set of values $t_1, t_2, \ldots, t_k$.  Our first step toward proving Theorem \ref{t.brownianlimit} will be to show that for any such $t_1, t_2, \ldots, t_k$ this joint law converges in law to the law of the corresponding Brownian motion restricted to these values.

To begin, we claim that if $p < 1/2$ then $\mathbb E[|E|] = \infty$. We argue the contrapositive: that if $\mathbb E[|E|] < \infty$ then $p \geq 1/2$.  Lemma \ref{l.stackfraction} states that if $\mathbb E[|E|] < \infty$ then the fraction of $\H$ symbols among the top $n$ elements in $X(-\infty,0)$ a.s.\ tends to $.5$ as $n \to \infty$.  One may deduce from this (and the fact that, by stationarity, $X(-\infty,0)$ agrees in law with $X(-\infty, \lfloor t n \rfloor)$) that $n^{-1/2} \dis_{\lfloor t n \rfloor}$ must converge in probability to zero in this case; together with the bounds in Lemma \ref{l.wordlength} this implies that $\Var[\dis_n] = o(n)$ in this case, so that (by Lemma~\ref{l.pchivariance}) we must have $p \geq 1/2$.

Now, by Lemma~\ref{l.stackfraction}, we must have (since we are assuming $p< 1/2$) that as $n \to \infty$ the fraction of $\x$ symbols among the leftmost $n$ elements of $X(1,\infty)$ tends a.s.\ to zero.
Thus, the number of $\x$ symbols in $X(1,\lfloor t n \rfloor)$ is $o(n^{1/2})$ with probability tending to $1$ as $n \to \infty$.  Thus if $t_3 = t_1 + t_2$ then $A_{\lfloor t_3 n \rfloor}$ is equal to
$A_{\lfloor t_1 n \rfloor}$ plus something with the law of $A_{\lfloor t_2 n \rfloor}$ plus an error which is $o(n^{1/2})$ (which arises when we determine the values of $Y_j$ for the $X_j$ in the sequence that are equal to $\x$, and also take into account the $O(1)$ rounding error from the $\lfloor \cdot \rfloor$) with high probability.  Similarly, $A_{\lfloor t n\rfloor}$ is equal in law to the sum of $k$ independent copies of $A_{\lfloor \frac{t}{k} n \rfloor}$ plus a term that is $o(n^{1/2})$ with high probability.  From this it follows that any subsequential weak limit of the random variable $n^{-1/2} A_{\lfloor t n \rfloor}$ is infinitely divisible with the appropriate exponent and mean zero, hence a centered Gaussian {\em with some covariance matrix}; moreover, if we consider a finite collection of $t$ values, the corresponding limiting joint law has independent Gaussian increments.  Next, we claim that $\cou_{\lfloor \frac{t}{k} n \rfloor}$ and $\dis_{\lfloor \frac{t}{k} n \rfloor}$ converge separately to Brownian motions each with the correct variance.  Indeed, recalling the variance limit described in Lemma \ref{l.pchivariance}, this follows from the tightness of the random variable sequences $n^{-1} \dis_n^2$ and $n^{-1} \cou_n^2$, which is implied by the uniform decay bound of Lemma \ref{l.wordlength}.  This determines the diagonal elements of the diffusion covariance matrix for the limit of $n^{-1/2} A_{\lfloor t n \rfloor}$, but we still need to rule out the possibility of off-diagonal elements. However, it suffices to observe that the law must be symmetric under the operation that replaces all $\{\H, \h \}$ symbols with the corresponding $\{\C, \ch \}$ symbols, so that by symmetry the limiting covariance between the $\cou$ and $\dis$ components must be zero.
The extension from the finite-dimensional convergence result above to the stronger form of convergence claimed in the theorem statement follows from exactly same argument used to establish the $p \geq 1/2$ case in the proof of Lemma \ref{l.zerovariancemeansmaximumdnsmall} (noting that Lemma~\ref{l.translateprob} can be used to bound the oscillation of $\dis$ on small intervals).

\begin{remark} \label{r.generalization}
When $p \leq 1/2$, the conclusion of Theorem \ref{t.brownianlimit} holds for $t \in [0,\infty)$ even if we condition on an arbitrary time-zero burger stack $S_0$ in place of $X(-\infty, 0)$, because the number of $\x$ symbols in $X(1,n)$ is $o(n^{1/2})$ with high probability.  When $p > 1/2$ the argument combined with the monotonicity results shows that the convergence still holds if we condition on an initial stack that is {\em well balanced} in the sense that the fraction of hamburgers among its top $k$ elements tends to $1/2$ as $k$ tends to infinity.
\end{remark}








\section{Random planar maps} \label{s.bijection}

\subsection{Bijection} \label{ss.bijection}

We begin by recalling a few classical constructions.  A {\bf planar map} $M$ is a planar graph together with a topological embedding into the sphere (which we represent as the compactified complex plane $\mathbb C \cup \{\infty \}$).   Self loops and edges with multiplicity are allowed.  Some of the first enumeration formulas for planar maps were given by Mullin and by Tutte in the 1960's \cite{MR0205882, MR0218276} and since then a sizable literature on enumerations of planar maps of various types (and various bijections with labeled trees, walks, pairs of trees, etc.) has emerged.  We describe only the very simplest formulations here.  Let $V=V(M)$ be the set of vertices of $M$ and $F=F(M)$ the set of faces.  Let $Q = Q(M)$ be the map whose vertex set is $V \cup F$, and whose edge set is such that each $f \in F$ is connected to all the vertices along its boundary (see Figure \ref{planarmapfigs}).  In other words, $Q$ is obtained from $M$ by adding a vertex to the center of each face and then joining each such vertex to all of the vertices (counted with multiplicity) that one encounters while tracing the boundary of that face. In particular $Q$ is bipartite, with the two partite classes indexed by $V$ and $F$, and all of the faces of $Q$ are quadrilaterals (with one quadrilateral for each edge of $M$).  Let $M'$ denote the dual map of $M$ (the map whose vertices correspond to the faces of $M$ --- an edge joins two vertices in $M'$ if an edge borders the corresponding faces in $M$).  See Figure \ref{planarmapfigs}.

\begin {figure}[!ht]
\begin {center}
\includegraphics [width=5in]{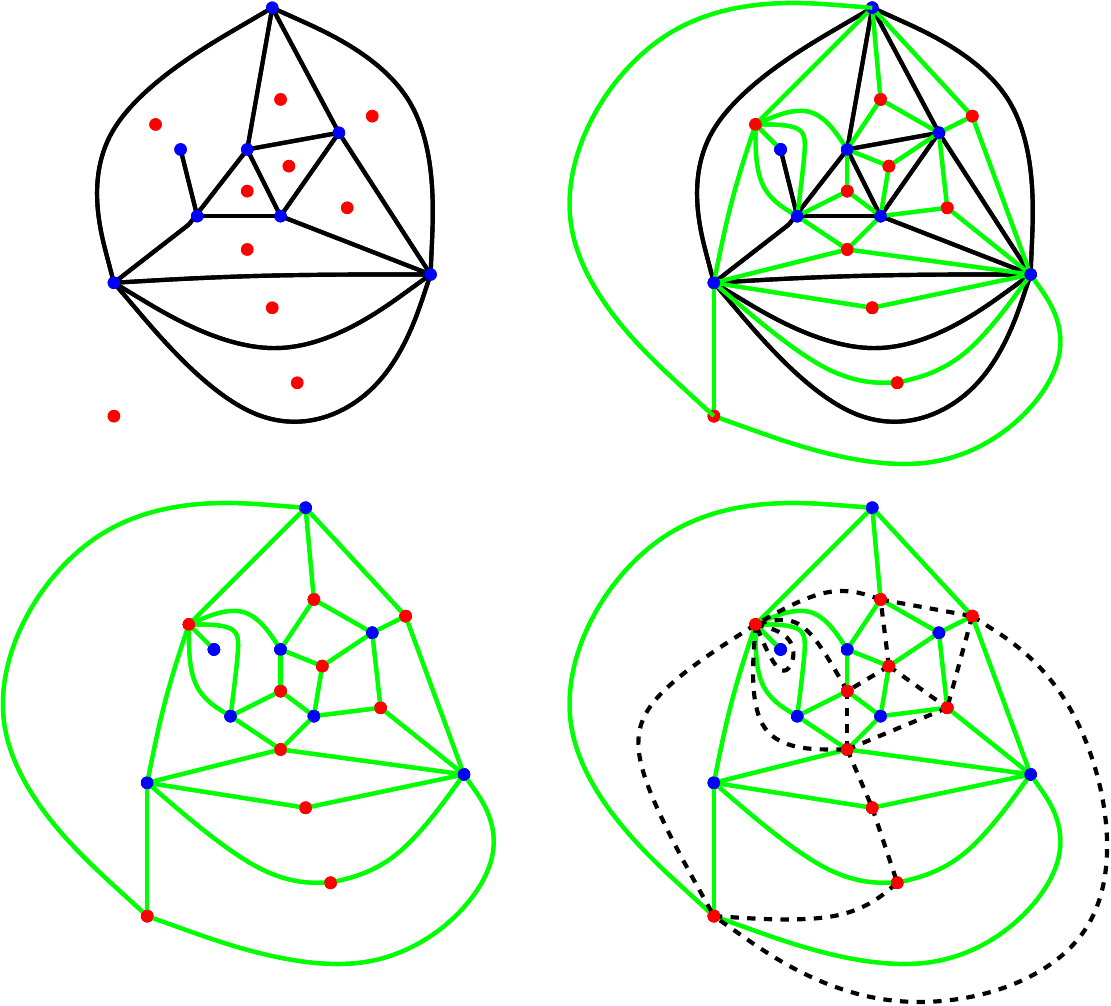}
\caption {\label{planarmapfigs} Upper left: a planar map $M$ with vertices in blue and ``dual vertices'' (one for each face) shown in red.  Upper right: the quadrangulation $Q=Q(M)$ formed by adding a green edge joining each red vertex to each of the boundary vertices of the corresponding face.  Lower left: quadrangulation $Q$ shown without $M$.  Lower right: the dual map $M'$ corresponding to the same quadrangulation, obtained by replacing the blue-to-blue edge in each quadrilateral with the opposite (red-to-red) diagonal.}
\end {center}
\end {figure}

\begin {figure}[!ht]
\begin {center}
\includegraphics [width=6in]{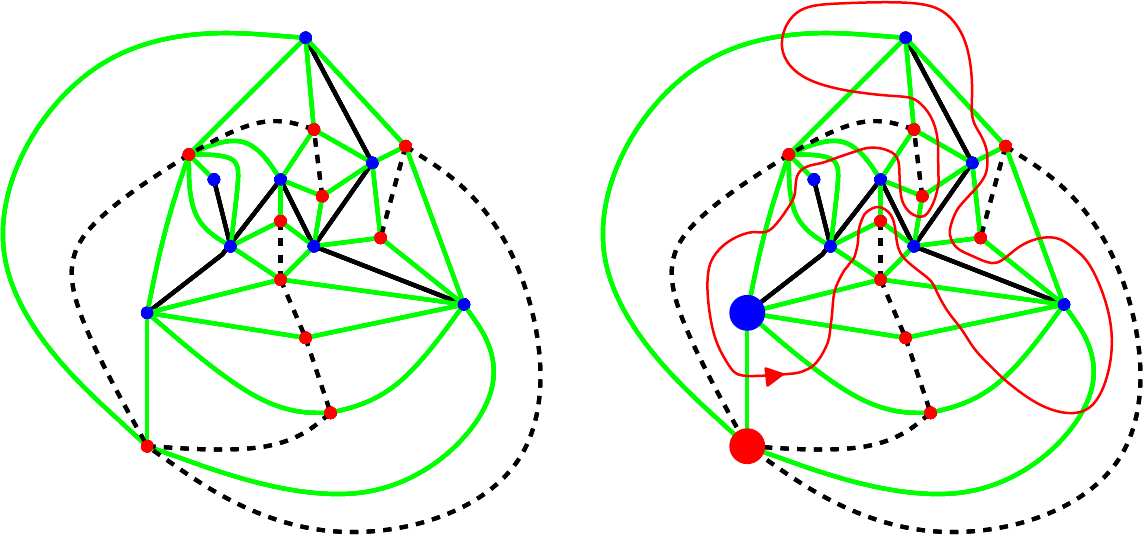}
\caption {\label{planarmapfigstree} Left: in each quadrilateral we either draw an edge (connecting blue to blue) or the corresponding dual edge (connecting red to red).  In this example, the edges drawn form a spanning tree of the original (blue-vertex) graph, and hence the dual edges drawn form a spanning tree of the dual (red-vertex) graph.   Right: designate a ``root'' (large blue dot) and an adjacent ``dual root'' (large red dot).  The red path starts at the midpoint of the green edge between the root and the dual root and crosses each of the green edges once, keeping the blue endpoint to the left and red endpoint to the right, until it returns to the starting position.}
\end {center}
\end {figure}

Now, suppose that $M$ is endowed with a distinguished oriented edge $e$.  This determines an oriented edge $e_0$ of $Q$ that has the same initial endpoint as $e$ and is the next edge clockwise (among all edges of $M$ and $Q$ that start at that endpoint) from $e_0$.  We refer to the endpoint of $e_0$ in $V$ as the {\em root} and the endpoint in $F$ as the {\em dual root}.  Now, suppose that $T$ is a subset of the edges of $M$ corresponding to a spanning tree.  Then the set $T'$ of dual edges to the edges in the complement of $T$ is necessarily a spanning tree of $M'$, see Figure \ref{planarmapfigstree}.  The union of $T$, $T'$ and $Q$ forms a triangulation, with each triangle containing two edges from $Q$ and one from either $T$ or $T'$.  Let $e_0,e_1, e_2, \ldots, e_{2n} = e_0$ be the sequence of edges of this triangulation hit by the path shown in Figure \ref{planarmapfigstree}, which starts at a midpoint of $e_0$ and then crosses each edge of $Q$ exactly once (with an element of $V$ on the left and an element of $F$ on the right) before returning to the initial edge.  This path goes through each triangle without ever crossing an edge of $T$ or $T'$ (and in a sense it describes the boundary between $T$ and $T'$).

For each $e_i$, let $d(e_i) = (d_1,d_2)$, where $d_1$ is the number of edges in the tree $T$ in between the $V$ endpoint of $e_i$ and the root, and $e_2$ is the number of edges in the tree $T'$ in between the $F$ endpoint of $e_i$ and the dual root.  Then the sequence $d(e_0), d(e_1), \ldots, d(e_{2n}) = d(e_0)$ is a simple walk on the lattice $\Z_+^2$ of non-negative integer pairs.  We can associate a corresponding word in the alphabet $\Theta$ by writing $\H$ or $\C$ each time the first or second (respectively) coordinate of $d(e_i)$ goes up and an $\h$ or $\ch$ each time the first or second (respectively) coordinate of $d(e_i)$ goes down.  The following is the word corresponding to Figure \ref{planarmapfigstree}: $$\C\C\C\H\C\H\H\ch\H\ch\ch\ch\C\C\h\H\ch\H\ch\C\C\h\C\h\h\ch\ch\h\H\h\h\ch.$$

In fact, this construction describes a bijection (essentially due to Mullin \cite{MR0205882} but more explicitly explained by Bernardi \cite{MR2285813}) between the set of pairs $(M,T)$ --- where $M$ is rooted planar map with $n$ edges and $T$ a distinguished spanning tree --- and the set of walks of length $2n$ in $\Z_+^2$ that start and end at the origin.  This set of walks is in turn equivalent to the set of length $2n$ words $W$ in the symbols $\{\C,\H,\ch,\h\}$ for which $\overline W = \emptyset$.

We can say more about this bijection.  Every quadrilateral of $Q$ corresponds to a burger. The quadrilateral is divided by an edge in $T \cup T'$ into two triangles.  The first triangle the path goes through corresponds to the step at which that burger was added to the stack, while the second corresponds to the step at which the same burger was ordered.  Quadrilaterals of $Q$ that are divided by $T$ edges correspond to hamburgers while elements divided by $T'$ edges correspond to cheeseburgers.  Another equivalent point of view is that every vertex of $T$ (besides the root) corresponds to a hamburger and every vertex of $T'$ (besides the dual root) corresponds to a cheeseburger (see Figure \ref{burgerchain2}).  Let $U_k$ be the union of the first $k$ triangles traversed.  Then a burger is added the first time the corresponding vertex is part of $U_k$, and ordered the first time the corresponding vertex lies in the interior of $U_k$.  The outer boundary vertices of $U_k$ (excluding the root and the dual root) represent the burgers on the stack at time $k$.  (The fact that $T$ and $T'$ are trees ensures that $U_k$ will be simply connected for all $k$.)

\begin {figure}[!ht]
\begin {center}
\includegraphics [height=3in]{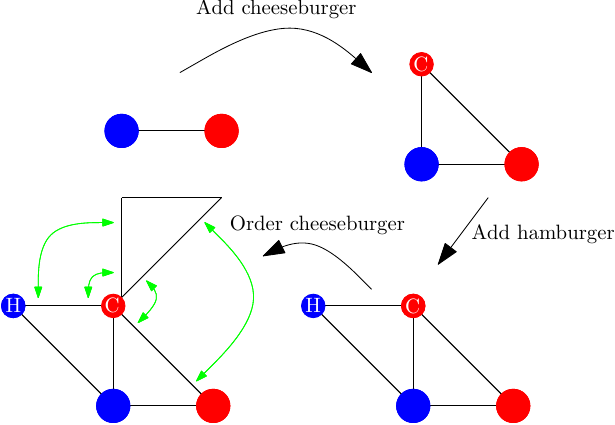}
\caption {\label{burgerchain2} Start with a single edge (upper left). We will recursively construct a triangulation with two types of vertices by ``exploring'' triangles one at a time (c.f.\ Angel's peeling process \cite{angelpeeling}). At each step the outer boundary of the explored region has two edges connecting two vertices of different types: one ``special edge''  (to which a new triangle will be glued) and one ``root edge'' (the initial edge in the construction; the special edge and root edge coincide at the beginning). When a new triangle is added, there are four possibilities for its ``type,'' corresponding to the four symbols.  Adding a triangle corresponding to one of the two ``burger'' symbols amounts to drawing a correspondingly colored vertex outside the surface and making a triangle by connecting it to the endpoints of the special edge (i.e., the clockwise-most boundary hamburger and the counterclockwise-most boundary cheeseburger). Adding a triangle corresponding to one of the ``order'' symbols amounts to making a new triangle by connecting one endpoint of the special edge to another vertex adjacent to the opposite endpoint of the special edge.  At any given time in this process, the small dots on the outer boundary of the explored region correspond to burgers that have not yet been consumed.  To put this somewhat fancifully, an illiterate restaurant owner armed only with triangular pieces of paper and cheeseburger and hamburger stickers could use this scheme to keep track of the day's events, gradually constructing a surface as the day goes on. }
\end {center}
\end {figure}

Now what happens if we remove the requirement that $T$ be a tree?  Let $T$ be any subset of the edges of $M$ and let $T'$ be the set of dual edges of the edges in the complement of $T$.  As shown in the left side of Figure \ref{planarmapfigsloops}, we still have a collection of loops, each of which passes through some subset of the triangles (crossing only edges of $Q$, no edges of $T$ or $T'$) such that every edge in $Q$ is crossed exactly once.

Observe carefully the loops on the left hand side of Figure \ref{planarmapfigsloops}.    We will now describe a way (corresponding to the right side of Figure \ref{planarmapfigsloops}) to modify the pair $T,T'$ (replacing some edges with dual edges, and vice versa) in a canonical way to obtain a tree/dual-tree pair $\tilde T, \tilde T'$.

Let $L_0, L_1, \ldots$ be the loops (as shown in Figure \ref{planarmapfigsloops}), where $L_0$ is the special loop that passes through the edge $e_0$.  Each loop can also be identified with the set of triangles it passes through and viewed as a subset of the set $\mathcal T$ of triangles in $T \cup T' \cup Q$.  Note that $\mathcal T$ can itself be viewed as a graph (two triangles adjacent if they share an edge).
Let $C_1, \ldots, C_k$ be the components of $\mathcal T \setminus L_0$.  (If we remove from the surface all of the triangles that $L_0$ passes through, then each $C_j$ is a connected subset of the triangles that remain.)  Clearly, $L_0$ passes through at least one triangle on the outer boundary of each $C_j$ (i.e., at least one triangle that shares an edge with a triangle in $C_j$).  Let $A_j$ describe the {\em last} triangle sharing an edge with the boundary of $C_j$ that is traversed by $L_0$.

This edge is either an edge of $T$ or an edge of $T'$.  In either case, we will now {\em replace it} with the opposite diagonal of the same quadrilateral (i.e., we replace a dual edge with an edge or an edge with a dual edge).  The effect of all of these replacements is to join one loop in each of the $C_j$ to the primary loop $L_0$.  The total number of loops has decreased by $k$.  We then repeat this process until at the end there is only a single loop dividing a tree from a dual tree, which we denote by $\tilde T$ and $\tilde T'$.  See Figure \ref{planarmapfigsloops}.  In that figure we colored all of the edges in $\tilde T \setminus T$ and $\tilde T' \setminus T$ a different color --- we refer to these as ``fictional edges'' since they do not belong to the original pair $T, T'$.

\begin {figure}[!ht]
\begin {center}
\includegraphics [width=6in]{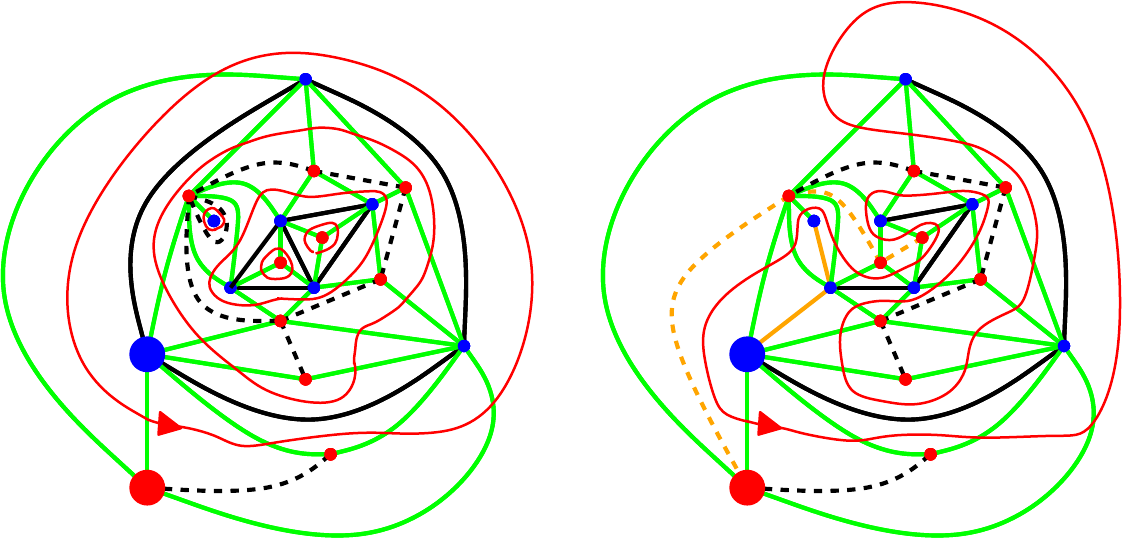}
\caption {\label{planarmapfigsloops} {\bf Left:} The subset of edges does not form a spanning tree, so we obtain {\em multiple} red loops (again, each green line is crossed once).  {\bf Right:} a canonical way to replace some edges with dual edges, and vice versa, to obtain a spanning tree.}
\end {center}
\end {figure}

To describe this $(M,T)$ by a word in $\Theta$, we first construct the word $\tilde W$ in $\{\C,\H,\ch,\h\}$ corresponding to $\tilde T$ and $\tilde T'$ via the method described above.  We then make an observation about $\tilde W$.  We already know from above that whenever the second triangle of a quadrilateral is traversed by the path corresponding to $\tilde T$ and $\tilde T'$, it corresponds to an order $\ch$ or $\h$.  We claim that whenever the second triangle of a ``fictional'' quadrilateral is traversed, it corresponds to the consumption of a burger on the top of the stack.  To see this, let $C_1, \ldots, C_k$ be as above and consider the boundary edges of one of the $C_j$.  Since these boundary edges cannot include any edges of $Q$, they consist only of edges in $T$ or $T'$.  Since the boundary of $C_j$ is connected, it contains either only edges of $T$ (and vertices of $V$) or only edges of $T'$ (and vertices of $F$).  Now, there is one edge or dual edge on the boundary of $C_j$ that is altered --- replaced by a fictional edge or dual edge --- to create two new triangles.  The reader may observe that in between the time that the final path goes through the first of these triangles and the time it goes through the second triangle, it does not traverse the first half of any other quadrilateral without traversing also the second half.  In other words, no burger added after the first triangle is traversed remains on the stack when the second triangle is traversed.

Now to define the word $W$ that corresponds to $(M,T)$, we start with $\tilde W$ and then put an $\x$ in place of each $\ch$ or $\h$ that corresponds to a triangle with a fictional edge or dual edge.  For example, the word corresponding to Figure \ref{planarmapfigsloops} is $$\C\H\ch\H\C\C\C\h\C\C\C\h\ch\H\H\ch\H\ch\ch\H\ch\C\C\h\h\x\h\x\H\x\x\x.$$
As noted above, $\tilde W$ corresponds to the same burger production/consumption sequence as $W$.  Given a sequence of this type, it is straightforward to reconstruct the corresponding planar map.  Namely, we first replace each $\x$ with the corresponding $\ch$ or $\h$, then construct the corresponding map with $\tilde T$ and $\tilde T'$, and then reverse the edges and dual edges corresponding to $\x$ symbols to create $T$ and $T'$.  Essentially the same argument as above shows that we recover the same $\tilde T$ and $\tilde T'$ when we follow the algorithm to create them from $T$ and $T'$.  We have now given a bijection between length $2n$ words $W$ in $\Theta$ with $\overline W = 0$ and pairs $(M,T)$, where $M$ is a rooted planar map with $n$ edges and $T$ is {\em any} distinguished subset of the edges of $M$.  In the case that $M$ is a fixed planar map, this is equivalent to a construction given by Bernardi in Section 4 of \cite{MR2438581}, as alluded to in the introduction.  Specifically, Section 4 of \cite{MR2438581} describes a mechanism for constructing $\tilde T$ from $T$ that is equivalent to the one described above.

Each loop $L \not = L_0$ has an inside (the component of the sphere minus $L$ that contains $L_0$) and an outside (the component containing $L_0$).  There are two types of loops $L$: those that pass only through triangles with edges on the inside (dual edges on the outside) and those that pass only through triangles containing edges on the outside (dual edges on the inside).  In other words, each loop can be viewed as a boundary interface between a cluster and a dual cluster, and the cluster can be either inside or outside of the loop.  Since every cluster (dual cluster) has a unique loop tracing its outer boundary, we can count the clusters (dual clusters) by counting the number of loops of the two types, which in turn corresponds to counting the number of $\x$'s matched to $\C$'s or to $\H$'s, as explained in the following table.

\vspace{.2 in}
\begin{center}
\begin{tabular} {|l|l|}
\hline
{\bf ROOTED MAP/EDGE SET: $(M,T)$} & {\bf CORRESPONDING WORD: $W$} \\
\hline
Number of edges in $\tilde T$ & Number of $\H$'s in $W$.  \\
\hline
Number of edges in $\tilde T'$ & Number of $\C$'s in $W$. \\
\hline
Number of loops &  $1$ plus number of $\x$'s in $W$. \\
\hline
Number of components of $T$ &  $1$ plus number of $\x$'s matched to $\C$'s. \\
\hline
Number of components of $T'$ &  $1$ plus number of $\x$'s matched to $\H$'s. \\
\hline
Number of edges in $T$ & Number of $\H$ to $\h$ or $\C$ to $\x$ matches.  \\
\hline
Number of edges in $T'$ & Number of $\C$ to $\ch$ or $\H$ to $\x$ matches. \\
\hline
\end{tabular}
\end{center}
\vspace{.2 in}

Part of our motivation for introducing the bijection of this section is that a complicated ``non-local'' quantity on the planar map side (number of loops) becomes a more straightforward ``local'' quantity on the $\Theta$-word side (number of $\x$ symbols).

\subsection{Critical Fortuin-Kastelyn model and infinite volume limits}

Now, suppose we fix a $p$, choose an $X(n)$ sequence from $\mu$, and then condition on $\overline W = 0$ where $W = X(1) \ldots X(2n)$.  Using the bijection above, this gives us a random rooted planar map $M$ decorated by a subgraph $T$.  The probability of seeing a particular $(M,T)$ with $\ell$ loops (which corresponds to $\ell$ of the $\x$ characters) is then proportional to \begin{equation} \label{e.pell} \Bigl( \frac{p/2}{(1-p)/4} \Bigr)^\ell = (2p/(1-p))^\ell.\end{equation}
If for some $q>0$ we write $p = \frac{\sqrt q}{2 + \sqrt q}$, so that $p$ solves the equation $2p/(1-p)=\sqrt q$, then \eqref{e.pell} becomes equivalent to $\sqrt{q}^\ell = q^{\ell/2}$.

It is natural to choose a random $(M,T)$ pair from the uniform probability measure on such pairs (with $n$ total edges in $M$) weighted by a quantity proportional to $q^{\ell/2}$.  One reason this is natural is that (once we condition on $M$) the law of $T$ is that of the self-dual Fortuin-Kastelyn model, which is in turn related to the critical $q$-state Potts model, see \cite{0305-4470-9-3-009}.  We will not say more about this here, but the survey \cite{MR2065722} contains one clear account of this connection (as well as conjectures, due to Duplantier, Nienhuis, and others, relating these models to SLE and conformal loop ensembles with $q = 2 + 2\cos(8\pi/\kappa)$ when $q \in (0,4)$ and $\kappa \in (4,8)$).

Next we claim that \begin{equation}\label{e.logprobtozero} \lim_{n \to \infty} \frac{1}{n} \log \mu \{ X(1,n) = \emptyset \} = 0.\end{equation}
In other words, the probability that $X(1, n) = \emptyset$ tends to zero more slowly than any exponential function of $n$.  (In fact,  Gwynne and Sun \cite{finitevolumeestimates} have recently proved that it decays like a negative power of $n$, and have derived the exact exponent as a function of $p$, but we will not need this here.)

Let $m = \lfloor \sqrt{n} \rfloor$ and note that it follows from Theorem \ref{t.brownianlimit} (using the type of argument in Lemma \ref{l.wordlength}) that $X(m, n-m)$ has a probability bounded below (independently of $n$) of containing at most $m$ elements.  Conditioned on this, there is some positive probability that the values of $X(1), \ldots, X(m)$ and $X(n-m+1) \ldots X(n)$ are such that $X(1,n) = \emptyset$.  This conditional probability cannot be smaller than exponentially small in $m$, and since $m=o(n)$, this implies \eqref{e.logprobtozero}.

Now we recall Cramer's theorem (see, e.g., the reference text \cite{dembo2009large}). Suppose that $Z_1, Z_2, \ldots$ are i.i.d.\ random variables each of which takes values in $\{1, 2, \ldots, k \}$ with resp.\ probabilities $a_1, a_2, \ldots, a_k$.  Write $a = (a_1, a_2, \ldots, a_k)$ for the probability vector. Let $u^n$ be a vector encoding the empirical distribution of $Z_1, Z_2, \ldots, Z_n$. That is, the $j$th component of the vector $u^n$, which we may denote by $u^n_j$,  is the fraction of the $n$ elements $Z_1, Z_2, \ldots, Z_n$ that are equal to $j$. Then Cramer's theorem states that if $A$ is any neighborhood in $\mathbb R^k$ about the point $a$, then $\mathbb P \{u^n \not \in A \}$ decays exponentially as $n \to \infty$.

Now, fix some large value for $m$ and let $Z_j = (X_{j m}, X_{j m + 1}, X_{j m + 2}, \ldots, X_{j m + m-1})$. That is, the $Z_j$ correspond to length-$m$ blocks of elements from the sequence $X_{\cdot}$. Now, if the $X_1, X_2, \ldots$ are i.i.d.\ then the corresponding $Z_1, Z_2, \ldots$ are i.i.d.\ and Cramer's theorem implies that $P\{u^n \not \in A \}$ decays exponentially as $n \to \infty$.  But \eqref{e.logprobtozero} implies that the $\mathbb P \{ X(1,mn+r) = \emptyset \}$ decays {\em subexponentially} in $n$ (uniformly over values of $r \in \{0,1,2,\ldots, m-1 \}$). This implies that if we condition on $X(1,mn+r) = \emptyset$ then (for any open set $A \subset \mathbb R^k$, where $k$ is the number of possible values $Z_1$ as defined above can assume, and $A$ contains the probability vector whose $\ell$th entry is the probability that $Z_1$ assumes the $\ell$th of its possible values) the conditional probability that $u^n \not \in A$ tends to zero as $n \to \infty$ (uniformly over choices for $r$).  

From this we may conclude that if one chooses $X(1) \ldots X(N)$ conditioned on $X(1,N) = \emptyset$ and then picks a random $k$ and recenters the sequence at $X(k)$, then one obtains a sequence of random recentered processes whose infinite volume limit is a random process with law $\mu$.  
In this sense, $\mu$ describes an infinite volume limit of the critical Fortuin-Kastelyn models.  Given a sequence $X(\cdot)$ sampled from $\mu$, the corresponding sequence $Y(\cdot)$ can be then used to construct an infinite random surface with an infinite spanning tree and spanning dual tree (in the manner described above for finite words).
From this perspective, Theorem \ref{t.brownianlimit} can be understood a scaling limit result about this random pair of infinite trees (which in turn encode the structure of a random infinite planar map).


\appendix

\section{Background: context and motivation} \label{s.background}
We now say a bit more about the larger project that motivated this work, namely, the problem of relating discrete random surfaces to continuum objects like Liouville quantum gravity, conformal loop ensembles, and the Schramm-Loewner evolution.

A longstanding open physics-motivated problem is to show that as the number of edges in the random maps described here tends to infinity, the corresponding loop-decorated surfaces (appropriately rescaled) converge in law to a limiting loop-decorated random surface $\mathcal M$.
The physics literature provides ample heuristic evidence, in various settings, for the existence of a phase transition analogous to the $p=1/2$ transition we present: $p > 1/2$ corresponds to what is called the {\em branched polymer} (a.k.a.\ {\em continuum random tree}) phase, while $p < 1/2$ is the {\em Liouville quantum gravity} phase.  The literature is too vast for us to properly survey here, but sample works in this direction include \cite{MR788098,0295-5075-7-8-009, Baillie199244, Thorleifsson1993787, 1988PhRvL..61.1433D, 1989NuPhB.326..583K, 1989MPLA....4..217K, 1990NuPhB.340..491D, MR1281359, Daul-1995, MR1432821, 1999PhLB..463..273E, 2006math.ph..11087E, 2009arXiv0909.1695B, 2011JSMTE..01..010B}.  This paper presents a clear and rigorous illustration of the phase transition.

The $p>1/2$ case is relatively simple.  Applying Theorem \ref{t.brownianlimit} when $p > 1/2$ we see that in the scaling limit, the spanning tree and the dual spanning tree that we construct converge in law to a.s.\ identical continuum random trees.  The collection of quadrilaterals corresponding to the tree and the collection corresponding to the dual tree can each be interpreted as approximations of the same continuum random tree: gluing the two together produces a surface that, as a metric space, should approximate the same tree.  (Note that since $X(-\infty,n)$ corresponds to a branch of a tree and $X(n,\infty)$ a branch of the corresponding dual tree, Proposition \ref{p.bottleneck} implies that when we glue these surfaces together a positive fraction of the points along any branch of the tree or dual tree correspond to length-one loops that are ``bottlenecks'' of the combined surface.  We do not know whether this remains true when $p=1/2$.)

The $p<1/2$ case is more interesting.  Based on various theorems and heuristics, we expect the limiting surface $\mathcal M$ to be a {\em Liouville quantum gravity (LQG)} surface decorated by a {\em conformal loop ensemble (CLE)}, with respective parameters $\gamma$ and $\kappa$ depending on $q$.   Precise formulations of this conjecture, along with definitions of LQG surfaces (whose laws depend on a parameter $\gamma$) and CLEs (whose laws depend on a parameter $\kappa$, and are based on the Schramm-Loewner evolution with parameter $\kappa$), can be found in \cite{sheffield-2006, 2008arXiv0808.1560D, 2010arXiv1012.4797S, 2010arXiv1012.4800D} (along with much more extensive lists of references; the CLEs relevant to this paper have a parameter $\kappa \in (4,8]$, while the $\kappa \leq 4$ case is discussed in \cite{2010merged}).  But in a sense this question has been around since the 1980's when Polyakov first introduced Liouville quantum gravity \cite{MR623209, MR623210} (see also \cite{MR947880}).

Polyakov writes in a recent memoir \cite{2008arXiv0812.0183} that he first became convinced of the connection between the discrete models and Liouville quantum gravity after jointly deriving, with Knizhnik and Zamolodchikov, the {\em KPZ formula} for so-called LQG scaling dimensions and comparing these dimensions with known combinatorial results for the discrete models \cite{MR947880}.  The KPZ formula was recently proved mathematically in \cite{2008arXiv0808.1560D, MR2501276} and can indeed be interpreted as evidence for the convergence conjectures.  But actually proving these conjectures, even in a physics sense, remains a challenge.

There are various ways to formulate a convergence statement, depending on what topology one uses when talking about convergence in law \cite{2010arXiv1012.4797S}.  One approach is to consider a model without loops and to interpret the discrete surfaces as random metric spaces.  If one uses the Gromov-Hausdorff topology on metric spaces, then recent independent works by Miermont and Le Gall show that these metric spaces converge in law to a limiting random metric space called the {\em Brownian map} \cite{2011arXiv1104.1606M, 2011arXiv1105.4842L}.  Prior to these recent works, it was known that these discrete random metric spaces converge in law to limiting random metric spaces along subsequences, and that any subsequential limit is a random metric space that is almost surely homeomorphic to the sphere; see \cite{MR2336042, 2008LeGall, MR2399286, MR2438999}, as well as the recent survey \cite{2011arXiv1101.4856L}.  The Brownian map is conjecturally equivalent, as a metric space, to an LQG surface with $\gamma = \sqrt{8/3}$ \cite{2010arXiv1012.4797S}.\footnote{One barrier to establishing this equivalence is that it remains an open question even to show that LQG surfaces {\em have} a canonical metric space structure. There is an obvious way to define a two-point distance function on an LQG surface in terms of limits, as explained in a footnote in \cite{2010arXiv1012.4797S}, but it has not been proved that these limits exist and are non-trivial.}  So far, there has been little success in extending the metric space theory to more general FK-weighted random surfaces, which are believed to correspond to other values of $\gamma$.

Another approach to the convergence problem is to ignore metric space structure and focus instead on conformal structure.  In this case, one may consider a conformal map from a random discretized surface to the sphere, and then study the induced area measure on the sphere and the images of the loops on the sphere.  The conformal point of view plays a central role in the physics literature on Liouville quantum gravity and ``random metrics'', which has been developed as part of a general subject called conformal field theory. (In this literature, the term ``metric'' refers most directly to a Riemannian metric tensor, not a two-point distance function.)  Various precise conjectures along these lines are presented in \cite{2008arXiv0808.1560D, 2010arXiv1012.4797S}, along with a more extensive bibliography.

This paper takes a third approach.  We encode loop-decorated surfaces by walks on $\mathbb Z^2$ (``inventory trajectories'').  Roughly speaking, the loops determine a canonical non-self-intersecting ``exploration path'' which goes through every face on the surface and more or less traces all the loops.  This path divides the surface into a ``tree'' and a ``dual tree'' --- and the exploration driving function encodes the structure of this pair of trees.  If we consider a topology in which two loop-decorated surfaces are considered close when their exploration driving functions are close, then Theorem \ref{t.brownianlimit} can be interpreted as a scaling limit result for the random surfaces.  It describes the limiting law of the pair of trees and the manner in which they are glued together.   (Note: the symmetry between the tree and dual tree makes our approach very different from the version of the Schaeffer bijection \cite{MR1465581} used to define the Brownian map \cite{2011arXiv1101.4856L}.  The latter involves a breadth-first search tree with geodesic branches and a dual tree that looks completely different from the tree itself --- in particular, its diameter scales like a different power of the number of edges.)

It turns out that if one replaces discrete loops with a continuum CLE, then the loops determine a continuum (space-filling) analog of the ``exploration path'' described above, which traces through all of the loops in a canonical way \cite{sheffield-2006, ms2013imag4}.  We can interpret this continuum exploration process as tracing the boundary between a continuum (space-filling) tree and a continuum (space-filling) dual tree.  If one draws these trees on a Liouville quantum gravity random surface, then the lengths of branches of the tree are well defined a.s\ by the results of \cite{2010arXiv1012.4797S}, so that one can construct a continuum analog of the exploration driving function described in this paper.

As mentioned in the introduction, since this paper was originally posted the arXiv there have been several additional papers that have extended the theory developed here. One recently completed work with Bertrand Duplantier and Jason Miller \cite{dms2014mating} demonstrates that LQG surfaces decorated by CLEs indeed have well-defined exploration driving functions themselves, and that their laws are those of precisely the same kinds of Brownian motions as long as the parameters $p,q,\kappa,\gamma$ are matched up correctly.\footnote{Specifically, $p\in[0,1/2)$, $q \in [0,4)$, $\kappa \in (4,8]$, $\gamma \in [\sqrt{2},2)$ and $q= 2 + 2 \cos \frac{8\pi}{\kappa}$, $\gamma = \sqrt{16/\kappa}$, $p= \frac{\sqrt q}{2 + \sqrt q}$.}    As in the discrete case, the continuum exploration driving function should encode the metric space structure of a certain tree (here a continuum random tree) and a dual-tree.  The simplest case is when $\kappa = 8$, corresponding to the case $p=0$.  The results of this paper are trivial (following from the classical central limit theorem for simple random walks) in that case, but the Liouville quantum gravity construction is still interesting.

As explained in \cite{dms2014mating}, this program allows us to interpret the results presented here to mean that the discrete loop-decorated surfaces of this paper converge to CLE-decorated Liouville quantum gravity in a topology where two configurations are close if the corresponding driving functions are close. Extensions of this result to stronger topologies (which encode lengths and intersection locations of loop configurations) have been established in works by (various subsets of) Cheng Mao, Ewain Gwynne, Jason Miller, and Xin Sun \cite{gwynnemaosun, finitevolumeestimates, strongertopology, finitevolumelimit}. See also the related results in \cite{berestyckilaslierray, 2015arXiv150201013C}.

Moving from this kind of convergence to convergence in topologies that more directly encode the metric space and/or conformal structure of the random surfaces themselves appears difficult, but conceivably possible.  One way to do this would be to show that if one couples a discrete loop-decorated surface with a CLE-decorated LQG surface in such a way that their driving functions are likely to be close, then it is likely that the two surfaces are also in some sense close as metric spaces, or as Riemannian manifolds with a conformal structure.   We expect statements of this kind to be true, but proving them will require new ideas.

Since we would ultimately like to extend the results presented here to other topologies, we conclude with one relevant question.  If two discrete loop-decorated surfaces $\mathcal M_1$ and $\mathcal M_2$ are conditioned to have exploration driving functions that are close in the $L^\infty$ sense for all time, is it then the case that (with high probability) one can conformally map $\mathcal M_1$ to $\mathcal M_2$ in such a way that the image of the exploration path in $\mathcal M_1$ is close to the exploration path in $\mathcal M_2$ (at least on a fixed compact interval of time)?  This question is already interesting in the case that $p=0$.  The results by Gill and Rohde in \cite{2011arXiv1101.1320G} about Brownian motion on these random surfaces, and the parabolicity of the infinite volume surfaces, constitute a promising step in this direction.

\bibliographystyle{halpha}
\bibliography{hamburger}
\end{document}